\documentclass[10pt,a4paper]{amsart}

\usepackage{amssymb}
\usepackage{hyperref}
\usepackage{enumitem}
\newtheorem{thm}{Theorem}[section]
\newtheorem{prop}[thm]{Proposition}
\newtheorem{lem}[thm]{Lemma}

\newtheorem{conj}[thm]{Conjecture} 
\newtheorem{cond}[thm]{Condition}

\theoremstyle{definition}
\newtheorem{definition}[thm]{Definition}

\theoremstyle{remark}
\newtheorem{remark}[thm]{Remark}

\numberwithin{equation}{section}

\newcommand{\R}{\mathbb{R}}  % The real numbers.
\newcommand{\Q}{\mathbb{Q}} % The rational numbers.
\newcommand{\Z}{\mathbb{Z}} % The integers
\newcommand{\C}{\mathbb{C}} % The complex numbers.
\newcommand{\h}{\mathbb{H}} % The upper half plane
\newcommand{\G}{\mathbb{G}_\mathrm{m}} % The multiplicative group
\newcommand{\GL}{\mathrm{GL}_2^+(\Q)} % Rational matrices with positive determinant
\newcommand{\SL}{\mathrm{SL}_2(\Z)} % Modular group
\newcommand{\GLR}{\mathrm{GL}_2^+(\R)} % Rational matrices with positive determinant
\newcommand{\alg}{\overline{\Q}} % algebraic numbers
\newcommand{\falt}{h_\mathrm{F}} %Faltings height
\DeclareMathOperator{\re}{Re} % Real part
\DeclareMathOperator{\im}{Im} % Imaginary part
\DeclareMathOperator{\codim}{codim} % Codimension
\DeclareMathOperator{\defect}{\delta} % defect
\begin{document}
	
	\title{Multiplicative independence of modular functions}
	\author{Guy Fowler}
	\address{Mathematical Institute, University of Oxford, 
		Oxford, OX2 6GG, United Kingdom.}
	\email{\href{mailto:guy.fowler@maths.ox.ac.uk}{guy.fowler@maths.ox.ac.uk}}
	\urladdr{\url{https://www.maths.ox.ac.uk/people/guy.fowler}}
	\date{\today}
	\thanks{\textit{Acknowledgements}: I would like to thank Ehud Hrushovski and Jonathan Pila for useful discussions on the contents of this paper. I am also very grateful to the referee for their many helpful comments and corrections. This work was supported by an EPSRC doctoral scholarship.}
	%\subjclass[2010]{}
	
	\begin{abstract}
	We provide a new, elementary proof of the multiplicative independence of pairwise distinct $\GL$-translates of the modular $j$-function, a result due originally to Pila and Tsimerman. We are thereby able to generalise this result to a wider class of modular functions. We show that this class includes a set comprising modular functions which arise naturally as Borcherds lifts of certain weakly holomorphic modular forms. For $f$ a modular function belonging to this class, we deduce, for each $n \geq 1$, the finiteness of $n$-tuples of distinct $f$-special points that are multiplicatively dependent and minimal for this property. This generalises a theorem of Pila and Tsimerman on singular moduli. We then show how these results relate to the Zilber--Pink conjecture for subvarieties of the mixed Shimura variety $Y(1)^n \times \G^n$ and prove some special cases of this conjecture.
	\end{abstract}
	
	\maketitle

\section{Introduction}\label{sec:intro}

Let $j \colon \h \to \C$ be the modular $j$-function, where $\h$ denotes the complex upper half plane. A $j$-special point is a complex number $\sigma$ such that $\sigma = j(\tau)$ for some $\tau \in \h$ with $[\Q(\tau) : \Q]=2$. The $j$-special points, often called singular moduli, are the $j$-invariants of elliptic curves with complex multiplication. They are algebraic integers with many interesting arithmetic properties, see e.g. \cite{Cox89}. 

In \cite{PilaTsimerman17}, Pila and Tsimerman investigated the multiplicative properties of $j$-special points. In particular, they established the following result. (A set of numbers $\{x_1, \ldots, x_n\}$ is called multiplicatively dependent if there exist $a_1, \ldots, a_n \in \mathbb{Z}$, not all zero, such that $\prod x_i^{a_i}=1$.)

\begin{thm}\cite[Theorem~1.2]{PilaTsimerman17}\label{thm:singdep}
	Let $n \geq 1$. There exist only finitely many $n$-tuples $(\sigma_1, \ldots, \sigma_n)$ of distinct $j$-special points such that the set $\{\sigma_1, \ldots, \sigma_n \}$ is multiplicatively dependent, but no proper subset of $\{\sigma_1, \ldots, \sigma_n \}$ is multiplicatively dependent.
\end{thm}

Observe that the independence of proper subsets and the distinctness of the $\sigma_i$ are required to avoid trivialities. Pila and Tsimerman proved Theorem~\ref{thm:singdep} by an o-minimal counting argument; in particular, the result is ineffective. The critical new ingredient in their proof was a functional multiplicative independence result for pairwise distinct $\GL$-translates of the $j$-function. Here and throughout $\GL$ and its subgroups act on $\h$ by M\"{o}bius transformations. 

Functions $f_1, \ldots, f_n \colon \h \to \C$ are called multiplicatively dependent modulo constants if some relation $\prod f_i^{a_i} = c$ holds for $a_i \in \mathbb{Z}$, not all zero, and $c \in \mathbb{C}$; if no such relation holds, then $f_1, \ldots, f_n$ are multiplicatively independent modulo constants. The functional independence result of Pila and Tsimerman was the following.

\begin{thm}\cite[Theorem~1.3]{PilaTsimerman17}\label{thm:jind}
	Let $g_1, \ldots, g_n \in \GL$. Suppose that the functions $j(g_1z), \ldots, j(g_nz)$ are pairwise distinct. Then $j(g_1z), \ldots, j(g_nz)$ are multiplicatively independent modulo constants.
\end{thm}

The proof of Theorem~\ref{thm:jind} in \cite{PilaTsimerman17} is via an elaborate tree argument. In particular, the method there does not readily generalise to other modular functions. A modular function is a meromorphic function $f \colon \h \to \C$ that is invariant under the action of $\SL$ on $\h$ and also ``meromorphic at the cusp''. The condition of $\SL$-invariance implies that $f$ has a Fourier expansion in terms of the nome $q = e^{2 \pi i z}$; being ``meromorphic at the cusp'' is then equivalent to this Fourier series having the form 
\[f(z)=\sum_{n=-m}^{\infty} a(n)  q^n,\]
for some $a(n) \in \C, m \in \Z$. The function $j$ is a Hauptmodul for the modular functions; that is, a function $f$ is a modular function if and only if $f$ may be written as a rational function (with coefficients in $\C$) of $j$ (see \cite[Theorem~11.9]{Cox89}).

In this paper, we provide a new proof of Theorem~\ref{thm:jind}. This uses just elementary properties of $j$. Notably, the proof generalises in a straightforward way to a wide class of modular functions, and so we establish Theorem~\ref{thm:modind}, a generalisation of Theorem~\ref{thm:jind}. To state this result, we make the following definition. Denote by $F_j$ the standard fundamental domain for the action of $\SL$ on $\h$, given by 
\[F_j = \Big\{ z \in \h \colon -\frac{1}{2} < \mathrm{Re}(z) \leq \frac{1}{2}, \lvert z \rvert \geq 1, \mbox{ and } \lvert z \rvert > 1 \mbox{ for } -\frac{1}{2} < \re(z) < 0 \Big\}.\]
The map $j$ restricts to a bijection $j |_{F_j}
 \colon F_j \to \C$. Therefore, any non-zero modular function $f \colon \h \to \C$ has only finitely many zeros and poles in $F_j$, and further if $f$ is also non-constant, then $f$ has at least one zero or pole in $F_j$.
 
\begin{definition}
	Let $f \colon \h \to \C$ be a non-constant modular function. Enumerate the zeros and poles of $f$ contained in $F_j$ as $w_1, \ldots, w_r$, where $\im(w_i) \leq \im(w_{i+1})$. We say that $f$ satisfies the divisor condition if $w_r + s$ is neither a zero nor a pole of $f$ for all $0 < s < 1$.
\end{definition}

\begin{remark}\label{rmk:div}
	If $\im(w_r) \geq 1$, then the divisor condition holds if $\im(w_r) > \im(w_{r-1})$. 
\end{remark}

Our generalisation of Theorem~\ref{thm:jind} is the following result.

\begin{thm}\label{thm:modind}
	Let $f \colon \h \to \C$ be a non-constant modular function satisfying the divisor condition. Let $g_1, \ldots, g_n \in \GL$. If the functions $f(g_1z), \ldots, f(g_nz)$ are pairwise distinct, then they are multiplicatively independent modulo constants.
\end{thm}

The divisor condition is satisfied generically by modular functions (though not by all such functions). In Section~\ref{sec:prod}, we show that the divisor condition is in particular satisfied by all of the (infinitely many) non-constant functions belonging to a certain multiplicative group of modular functions with product formulae which arise as Borcherds lifts of some weakly holomorphic modular forms. This gives us a natural set of examples of modular functions satisfying the divisor condition.

For $f \colon \h \to \C$ a modular function, define, in analogy with the case of $j$, an $f$-special point to be a complex number $\sigma$ such that $\sigma =f(\tau)$ for some $\tau \in \h$ with $[\Q(\tau) : \Q]=2$. If $f$ is a modular function, then there exists a rational function $R \in \C(t)$ such that $f(z) = R(j(z))$. The $f$-special points are then precisely the images under $R$ of the $j$-special points. In particular, $f$-special points correspond to the CM elliptic curves, viewing $f$ as a function on the moduli space of elliptic curves over $\C$ given by the modular curve $Y(1) = \SL \backslash \h$. 

If $f$ satisfies the divisor condition, then we are able to establish the following finiteness result on multiplicatively dependent tuples of $f$-special points, analogous to Theorem~\ref{thm:singdep}. 

\begin{thm}\label{thm:moddep}
	Suppose $f$ is a non-constant modular function satisfying the divisor condition. Let $n \geq 1$. There exist only finitely many $n$-tuples $(\sigma_1, \ldots, \sigma_n)$ of distinct $f$-special points such that the set $\{\sigma_1, \ldots, \sigma_n \}$ is multiplicatively dependent, but no proper subset of $\{\sigma_1, \ldots, \sigma_n \}$ is multiplicatively dependent.
\end{thm}

For $f \in \alg(j)$, the proof of this result is via a modification of Pila and Tsimerman's o-minimal counting argument. The extension to $f \in \C(j)$ proceeds by a specialisation argument. We note here that Theorem~\ref{thm:moddep} is ineffective.

Denote by $\G = \G(\C)$ the multiplicative group of complex numbers. An $n$-tuple $(\sigma_1, \ldots, \sigma_n)$ of $f$-special points is multiplicatively dependent if some product $\prod \sigma_i^{a_i}$, with the $a_i \in \Z$ not all zero, lies in the trivial subgroup $\{1\} \leq \G$. A natural extension of Theorem~\ref{thm:moddep} is then to consider, given a fixed subgroup $\Gamma \leq \G$, those $n$-tuples $(\sigma_1, \ldots, \sigma_n)$ of $f$-special points such that, for some $a_i \in \Z$ not all zero, one has $\prod \sigma_i^{a_i} \in \Gamma$. 

When $\Gamma \leq \G$ is of finite rank and $f \in \alg(j)$, we are able to extend Theorem~\ref{thm:moddep} to this setting. Recall that a subgroup $\Gamma \leq \G$ is said to be of finite rank if there exists $\Gamma_0 \subset \Gamma$ such that $\Gamma_0$ is a finitely generated subgroup of $\G$ and, for every $\gamma \in \Gamma$, there exists $m \geq 1$ such that $\gamma^m \in \Gamma_0$. We thereby establish the following result. Here a set $\{x_1, \ldots, x_n\}$ is called $\Gamma$-dependent, for $\Gamma \leq \G$, if there exist $a_1, \ldots, a_n \in \mathbb{Z}$, not all zero, such that $\prod x_i^{a_i} \in \Gamma$.

\begin{thm}\label{thm:gamdep}
	Suppose $f$ is a non-constant modular function satisfying the divisor condition and also $f \in \alg(j)$. Let $\Gamma \leq \G$ be of finite rank and $n \geq 1$. Then there exist only finitely many $n$-tuples $(\sigma_1, \ldots, \sigma_n)$ of distinct $f$-special points such that the set $\{\sigma_1, \ldots, \sigma_n \}$ is $\Gamma$-dependent, but no proper subset of $\{\sigma_1, \ldots, \sigma_n \}$ is $\Gamma$-dependent.
\end{thm}

The plan of this article is as follows. In Section~\ref{sec:elem}, we give a new proof of Theorem~\ref{thm:jind} and then generalise it to prove Theorem~\ref{thm:modind}. Section~\ref{sec:prod} establishes a natural class of modular functions, arising via Borcherds lifts, to which this theorem applies. In Section~\ref{sec:ax}, we introduce the setting of the mixed Shimura variety $Y(1)^m \times \G^n$ which is used for the remainder of the paper. The proof of Theorem~\ref{thm:moddep} is contained in Section~\ref{sec:tup}. In Section~\ref{sec:ZP}, we relate Theorem~\ref{thm:moddep} to the Zilber--Pink conjecture for $Y(1)^n \times \G^n$. The proof of Theorem~\ref{thm:gamdep} then takes place in Section~\ref{sec:finrk}. Finally, the Zilber--Pink context of Theorem~\ref{thm:gamdep} is considered in Section~\ref{sec:ZPgam}. 

We note here that several proofs in Sections~\ref{sec:tup}--\ref{sec:ZPgam} will use techniques from the study of o-minimality, in particular point counting arguments. By definable we will always mean definable in the o-minimal structure $\R_{\mathrm{an}, \exp}$. We identify (subsets of) $\C$ with (subsets of) $\R^2$ in the natural way. In addition, given a function $f \colon A \to B$ and $n > 1$, we will often, by an abuse of notation, also write $f$ for the function $A^n \to B^n$ given by the Cartesian power of $f$. It should always be clear from context which function is meant. 

\section{Multiplicative independence}\label{sec:elem}

In this section, we give a new proof of Theorem~\ref{thm:jind}, and then show how this method may be extended to prove the more general Theorem~\ref{thm:modind}.

\begin{proof}[Proof of Theorem~\ref{thm:jind}]

Let $g_1, \ldots, g_n \in \GL$, and suppose that the functions $j(g_1z), \ldots, j(g_n z)$ are pairwise distinct. Then the cosets 
\[[g_1], \ldots, [g_n] \in \mathrm{P}\SL \backslash \mathrm{P}\GL\]
are pairwise distinct. For $g \in \GL$, we may, as in the proof of \cite[Proposition~7.1]{Pila14a}, write $g = \gamma h$, where $\gamma \in \SL$ and $h z = r z + s$ for $r, s \in \Q$ with $0 < r$ and $0 \leq s < 1$. The cosets $[g_i] \in \mathrm{P}\SL \backslash \mathrm{P}\GL$ are pairwise distinct if and only if the corresponding linear functionals $r_i z + s_i$ are pairwise distinct. Different $g_i$ may have associated the same $r_i$, so reindex them as $g_{i,k}$, associated with the functional $r_i z + s_{i,k}$, where $r_1 < r_2 <\ldots <r_l$ and $s_{i,k}<s_{i,k'}$ for $k < k'$.

To prove Theorem~\ref{thm:jind}, it is enough to find $z \in \h$ such that $j(r_i z + s_{i,k})=0$ if and only if $(i, k)=(1,1)$. Recall that $j(\zeta_6)=0$, where $\zeta_6 = \exp( \pi i /3)$. Therefore, setting 
\[z = \frac{1}{r_1} (\frac{1}{2}-s_{1,1}) + \frac{1}{r_1}\frac{\sqrt{3}}{2}i \in \h,\]
 so that $r_1 z + s_{1,1} = \zeta_6$, gives that $j(r_1 z +s_{1,1})=0$. 
 
 It remains to show that $j(r_i z + s_{i,k}) \neq 0$ for $(i, k) \neq (1,1)$. To do this we use two elementary facts about $j$. First, that for $w$ with $1/2 < \re(w) < 3/2$, if $j(w)=0$, then $\im(w) < \sqrt{3}/{2}$. Second, that $j(w) \neq 0$ whenever $\im(w) > \sqrt{3}/{2}$. Either of these is clear by considering the tessellation of $\h$ by translates of the fundamental domain $F_j$ for the action of $\SL$, since the only zero of $j$ in $F_j$ is at $\zeta_6$.

For $k>1$, note that 
\[\re(r_1 z + s_{1,k}) = \re(r_1z + s_{1,1}) + \re(s_{1,k}-s_{1,1})= \frac{1}{2} + (s_{1,k}-s_{1,1}).\]
Since $0 \leq s_{1,1} < s_{1,k} <1$, we have that $1/2 < \re(r_1 z + s_{1,k}) < 3/2$. Therefore, $j(r_1 z + s_{1,k}) \neq 0$ by the first fact, since $\im(r_1 z + s_{1,k})= \sqrt{3}/2$. Now for $i>1$,
\[\im(r_i z + s_{i,k}) = r_i \im(z) = r_i \frac{1}{r_1} \frac{\sqrt{3}}{2} > \frac{\sqrt{3}}{2}\]
since $r_i > r_1$ for $i > 1$. Therefore, by the second fact, $j(r_i z + s_{i,k}) \neq 0$ for $i > 1$. Hence, $j(r_i z + s_{i,k}) \neq 0$ for $(i, k) \neq (1,1)$.
\end{proof}

Now we show how this method may be generalised to prove Theorem~\ref{thm:modind}. 

\begin{proof}[Proof of Theorem~\ref{thm:modind}]
	Let $f \colon \h \to \C$ be a non-constant modular function. So $f$ is a rational function of $j$. Let $g_1, \ldots, g_n \in \GL$, and suppose that $f(g_1 z), \ldots, f(g_n z)$ are pairwise distinct. In particular, the associated cosets $[g_i] \in \mathrm{P}\SL \backslash \mathrm{P}\GL$ are pairwise distinct, and so, as above, we may rewrite the $f(g_iz)$ as $f(r_i z + s_{i,k})$ for $r_i, s_{i,k} \in \Q$ with $r_i >0$ and $0 \leq s_{i,k}<1$. These linear functionals $r_i z + s_{i,k}$ are again pairwise distinct and we may assume $r_1 < r_2 <\ldots <r_l$ and $s_{i,k}<s_{i,k'}$ for $k < k'$.

The function $f$ is meromorphic on $\h$. To prove the multiplicative independence modulo constants of the functions $f(r_i z + s_{i,k})$, it will be enough to find $z \in \h$ such that $r_i z + s_{i,k}$ is either a zero or a pole of $f$ if and only if $(i,k)=(1,1)$. Since $f$ is a rational function of $j$, and $j|_{F_j} \colon F_j \to \C$ is bijective, the function $f$ has only finitely many zeros and poles in $F_j$. Further, $f$ has at least one zero or pole in $F_j$ since $f$ is non-constant. 

Enumerate the zeros and poles of $f$ in $F_j$ as $w_1, \ldots, w_r$, where $\im(w_i) \leq \im(w_{i+1})$. We may then proceed for $f$ as we did in the above proof for $j$, replacing $\zeta_6$ by $w_r$, provided that
\begin{enumerate}
	\item $w_r + s$ is neither a zero nor a pole of $f$ for all $0 < s < 1$, and
	\item $f$ has no zero or pole with imaginary part $> \im (w_r)$.
\end{enumerate}
The first of these is just the divisor condition. For the second, note that if $f$ has a zero or pole in $\h$ with imaginary part $> \im(w_r)$, then $f$ must have a zero or pole in $F_j$ with imaginary part $> \im(w_r)$, as may be seen by considering the tessellation of $\h$ by $\SL$-translates of $F_j$. This cannot happen by the definition of $w_r$, and so we are done.
\end{proof}

\begin{remark}\label{rmk:necessary}
	Suppose $f$ is a non-constant modular function which does not satisfy the divisor condition. Let $g_1, \ldots, g_n \in \GL$ be such that the functions $f(g_1 z), \ldots, f(g_n z)$ are pairwise distinct. For such $g_i$ suitably generic, there will still exist some $z \in \h$ such that $g_i z$ is either a zero or a pole of $f$ for exactly one $i$. Thus the translates $f(g_1 z), \ldots, f(g_n z)$ would still be multiplicatively independent modulo constants. The divisor condition is sufficient for this to be true for all possible choices of the $g_i$, and hence sufficient to establish Theorem~\ref{thm:modind}, but is not obviously necessary. Indeed, there does not appear to be an obvious obstruction to a corresponding functional independence result holding for an arbitrary non-constant modular function. It seems likely that there is a weaker condition which would still suffice to prove Theorem~\ref{thm:modind}.
\end{remark}

\section{Borcherds products for modular functions}\label{sec:prod}

In this section, we show that the divisor condition is satisfied by the non-constant elements of a natural class of modular functions, which arises as the set of Borcherds lifts of certain weakly holomorphic modular forms. We thereby establish multiplicative independence for non-constant functions belonging to this set. We introduce the following notation, after \cite{Borcherds95a}.

Let $A$ be the set of weakly holomorphic modular forms $f$ of weight $1/2$ on $\Gamma_0(4)$ which have a Fourier series of the form $f(z) = \sum c(n)q^n$, where $c(n) \in \Z$ are such that $c(n)=0$ unless $n \equiv 0, 1 \bmod 4$. (Weakly holomorphic means that we allow $f$ possibly to have poles at the cusps.) It is clear that the elements of the set $A$ form an additive group.

Let $B$ be the set of integer weight meromorphic modular forms for some character of $\SL$, all of whose zeros and poles are located at either cusps or imaginary quadratic numbers, and that have Fourier expansions with integer coefficients and leading coefficient $1$. The elements of $B$ form a multiplicative group.

Denote by $H(n)$ the Hurwitz class number of discriminant $-n$ when $n>0$ and set $H(0)=-1/12$. Then let the function $\tilde{H}$ be defined by
\[\tilde{H}(z) = \sum_{n\geq0} H(n)q^n = -\frac{1}{12} + \frac{q^3}{3}+\frac{q^4}{2}+q^7+\ldots.\]
Borcherds established the following isomorphism.

\begin{thm}\cite[Theorem~14.1]{Borcherds95}\label{thm:borch}
	Let $\Psi \colon A \to B$ be given by 
	\[f(z) = \sum a(n) q^n \mapsto \Psi(f(z)) = q^{-h} \prod_{n>0}(1-q^n)^{a(n^2)},\]
	where $h$ is the constant coefficient in the Fourier series of $f(z)\tilde{H}(z)$. Then $\Psi$ is an isomorphism from the additive group $A$ to the multiplicative group $B$. Further, the weight of the meromorphic modular form $\Psi(f)$ is equal to $a(0)$, and the multiplicity of the zero of $\Psi(f)$ at an imaginary quadratic number $\tau \in \h$ of discriminant $\Delta$ is equal to
	\[\sum_{n>0} a(n^2 \Delta).\]
	Here the discriminant $\Delta$ of an imaginary quadratic number $\tau$ is defined as $\Delta = b^2-4ac$, where $a,b,c \in \Z$ satisfy $a \tau^2 + b \tau + c= 0$ and $\gcd(a,b,c)=1$.
\end{thm}

A modular function is a weight $0$ meromorphic modular form for the trivial character of $\SL$. The modular functions in $B$ clearly form a subgroup of the group $B$, which we denote $B_0$. Applying Borcherds's isomorphism, we have that $B_0 \subset \Psi( A_0)$, where $A_0$ is the subgroup of $A$ comprising those functions $f \in A$ which have a Fourier expansion $f(z) = \sum a_f(n) q^n$ with $a_f(0)=0$.

In this section, we prove the following theorem.

\begin{thm}\label{thm:borchind}
	Let $f \in B_0$ be non-constant and $g_1, \ldots, g_n \in \GL$. Suppose that the functions $f(g_1 z), \ldots, f(g_n z)$ are pairwise distinct. Then the functions $f(g_1 z), \ldots, f(g_n z)$ are multiplicatively independent modulo constants.
\end{thm}

This theorem will follow from Theorem~\ref{thm:modind} once we show that non-constant functions $f \in B_0$ satisfy the divisor condition. To do this, we consider a basis for the (free abelian) group $A$. The group $A$ has a basis 
\[\{ f_d(z) : d \geq 0 \mbox{ and } d \equiv 0, 3 \bmod 4\},\]
where $f_d$ is the unique element of $A$ with a Fourier expansion of the form
\[ f_d(z) = q^{-d} + \sum_{D>0} A(d, D) q^D.\]
See, for example, \cite[Chapter~4]{Ono04a}, where the first few such functions are listed. The isomorphism of Theorem~\ref{thm:borch} then tells us that the functions $\Psi(f_d)$ form a basis for the group $B$. Recall that the weight of an element $\Psi(f) \in B$ is equal to the constant coefficient $a_f(0)$ in the Fourier expansion of $f = \sum a_f(n) q^n \in A$. The subgroup $B_0$ of modular functions in $B$ is therefore contained in $\Psi(A_0)$, which has a basis given by
\[\{ \Psi(f_d) : d >0 \mbox{ and } d \equiv 0, 3 \bmod 4\}.\]

\begin{lem}\label{lem:zero}
	Let $d>0$ satisfy $d \equiv 0, 3 \bmod 4$. Then there exists a simple zero $\tau^* \in F_j$ of $\Psi(f_d)$ with the property that, for every $\tau \in F_j \setminus \{\tau^*\}$ which is either a zero or a pole of $\Psi(f_d)$, one has that $\im(\tau) < \im(\tau^*)$. Let $\alpha_d = \im(\tau^*)$. The resulting sequence $\{\alpha_d : d>0, \, d \equiv 0,3 \bmod 4\}$ is strictly increasing in $d$.
\end{lem} 

\begin{proof}
Write $A(d,n)$ for the Fourier coefficients of $f_d$, so that
	\[f_d(z) = \sum_{n \in \Z} A(d,n) q^n.\]
	Then, in particular, $A(d, -d)=1$ and $A(d, n)=0$ for all $n<0, n\neq -d$. By Theorem~\ref{thm:borch} all the zeros and poles of $\Psi(f_d)$ in $F_j$ are located at imaginary quadratic numbers in $F_j$. Further, the multiplicity of the zero or pole of $\Psi(f_d)$ at an imaginary quadratic number $\tau \in F_j$ of discriminant $\Delta$ is equal to
	\[\sum_{n>0} A(d, n^2 \Delta).\]
	In particular, this is equal to zero if $\Delta < -d$ and equal to $1$ if $\Delta = -d$. Hence, $\Psi(f_d)$ has a simple zero at each imaginary quadratic number in $F_j$ of discriminant $-d$. Further, all zeros and poles of $\Psi(f_d)$ in $F_j$ are contained in the set
	\[ \{\tau \in F_j : \tau \mbox{ is imaginary quadratic of discriminant } \Delta \mbox{ with } \lvert \Delta \rvert \leq d\}.\]
	
	The quadratic imaginary numbers in $F_j$ of a given discriminant $\Delta$ can be explicitly described. Each corresponds to the unique (in $\h$) solution to an equation $a z^2 - b z + c= 0$, where $a,b,c \in \Z$ satisfy $\Delta = b^2 - 4 ac$, $\gcd(a,b,c)=1$, and either $-a < b \leq a < c$ or $0 \leq b \leq a = c$. The imaginary quadratic $\tau$ corresponding to such a triple $(a,b,c)$ is
	\[\tau = \frac{b + \sqrt{\Delta}}{2a},\]
	so $\im(\tau) = \lvert \Delta \rvert^{1/2}/2a$.
	
	For each discriminant $\Delta <0$, there exists a unique such triple $(a,b,c)$ with $a=1$. See \cite[Proposition~2.6]{BiluLucaMadariaga16}. Write $\tau_{\Delta}$ for the element of $F_j$ corresponding to this triple. Then $\tau_{\Delta}$ is a simple zero of $\Psi(f_{-\Delta})$. Now let $\tau \in F_j \setminus \{\tau_{\Delta}\}$ be either a zero or a pole of $f_{-\Delta}$. Then $\tau$ is an imaginary quadratic number of discriminant $\Delta'$, where $\lvert \Delta' \rvert \leq \lvert \Delta \rvert$. Then clearly $\im (\tau_{\Delta}) > \im (\tau)$. So for $d>0$ satisfying $d \equiv 0, 3 \bmod 4$, let $\Delta = -d$ and take $\tau^* = \tau_{\Delta}$. This proves the first part of the lemma, and the second part then follows immediately.	
	\end{proof}

Now we come to the proof of Theorem~\ref{thm:borchind}.

\begin{proof}[Proof of Theorem~\ref{thm:borchind}]
 Let $f \in B_0$ be non-constant. By Theorem~\ref{thm:modind}, it suffices to show that $f$ satisfies the divisor condition. Since $f$ is not the identity element of $B_0$, there are $0<d_1< \ldots <d_k$, satisfying $d_i \equiv 0,3 \bmod 4$, and $a_1, \ldots, a_k \in \Z \setminus \{0\}$ such that
	\[f = \Psi(f_{d_1})^{a_1} \cdots \Psi(f_{d_k})^{a_k}.\]
	Therefore,
	\[\{\mbox{zeros and poles of } f\} \subset \bigcup_{i=1}^{k} \{\mbox{zeros and poles of } \Psi(f_{d_i}) \}.\]
	
	By Lemma~\ref{lem:zero}, there is a simple zero $\tau^*$ of $\Psi(f_{d_k})$ which is the unique zero/pole in $F_j$ of maximum imaginary part for $\Psi(f_{d_k})$. Further $\tau^*$ has imaginary part greater than that of every zero or pole in $F_j$ of a function $\Psi(f_{d_i})$, for $i < k$. Since $a_k \neq 0$, the function $f$ therefore has a unique zero/pole (depending on the sign of $a_k$) in $F_j$ of maximum imaginary part. 
	
	Enumerate the zeros and poles of $f$ in $F_j$ as $w_1, \ldots, w_r$, where $\im(w_i) \leq \im(w_{i+1})$. We then have that $w_r = \tau^*$. By the previous paragraph and Remark~\ref{rmk:div}, $f$ satisfies the divisor condition if $\im(\tau^*)\geq 1$, which is true whenever $d_k \geq 4$. 
	
	If $d_k = 3$, then $f = \Psi(f_3)^{a}$ for some $a \in \Z \setminus \{0\}$. Thus $f = j^{a/3}$ since $j = \Psi(f_3)^3$ by \cite[\S14, Example~2]{Borcherds95}. Hence the divisor condition holds for $f$ since it holds for $j$, as was shown already in the proof of Theorem~\ref{thm:jind} in Section~\ref{sec:elem}. Hence, for all $d_k$, the divisor condition holds for $f$, and so we are done. 
	\end{proof}

\begin{remark}\label{rmk:ratfour}
	If $f$ is a modular function with rational Fourier coefficients, then $f \in \Q(j)$, see e.g. \cite[Proposition~12.7]{Cox89}. So in particular $f \in \alg(j)$ for all those functions covered by Theorem~\ref{thm:borchind}. Therefore the additional hypothesis in Theorem~\ref{thm:gamdep} that $f \in \alg(j)$ is not in fact a further restriction in this case. (On the other hand, the divisor condition alone does not imply that $f \in \alg(j)$, so in the fully general case Theorem~\ref{thm:gamdep} does impose an additional restriction on $f$.)
\end{remark}

\section{The mixed Shimura variety setting}\label{sec:ax}

The Diophantine problems we consider in Sections~\ref{sec:tup}--\ref{sec:ZPgam} may naturally be formulated as questions about a mixed Shimura variety. In this section, we introduce this setting, which we will use in the sequel. In particular, we formulate the Ax--Schanuel-type statements which will be required in the proofs of Theorems~\ref{thm:moddep} and \ref{thm:gamdep}. By an Ax--Schanuel statement we mean a functional transcendence result that has a similar form to Ax's theorem for the exponential function \cite{Ax71}.

Denote by $Y(1)$ the modular curve $\SL \backslash \h$, which we identify with $\C$ by means of the $j$-function, and by $\G = \G(\C)$ the multiplicative group of complex numbers. Let $X = X_{m,n}= Y(1)^m \times \G^n$, $U =U_{m,n}=\h^m \times \C^n$, and $\pi \colon U \to X$ be given by
\[\pi(z_1, \ldots, z_m, u_1, \ldots, u_n) = (j(z_1),\ldots, j(z_m), \exp(u_1), \ldots, \exp(u_n)).\]
Then $X$ is a mixed Shimura variety, uniformised by the map $\pi \colon U \to X$. We require the following definitions.

\begin{definition}\label{def:special}
	\hfill
	\begin{enumerate}
		\item A weakly special subvariety of $\G^n$ is a subvariety defined by a finite system of equations $\prod x_i^{a_{ij}} = c_j \in \G$, where the $a_{ij} \in \Z$ are such that the lattice generated by $(a_{1j}, \ldots, a_{nj})$ is primitive. It is a special subvariety if every $c_j$ is a root of unity.
		\item A weakly special subvariety of $\h^m$ is a complex-analytically irreducible component of the intersection with $\h^m$ of a subvariety of $\C^m$ defined by a collection of equations of the form $x_i = g x_j$ for $g \in \GL$ and $x_k = c$ for $c \in \h$ constant. It is a special subvariety if every constant coordinate $c$ is a quadratic imaginary number.
		\item A weakly special subvariety of $Y(1)^m$ is the image under $j$ of a weakly special subvariety of $\h^m$. It is a special subvariety if in addition every constant coordinate is a $j$-special point.
		\item A (weakly) special subvariety of $X$ is a product $M \times T$, where $M$ is a (weakly) special subvariety of $Y(1)^m$ and $T$ is a (weakly) special subvariety of $\G^n$.
		\item An algebraic subvariety of $U$ is a complex-analytically irreducible component of $Y \cap U$, where $Y \subset \C^m \times \C^n$ is an algebraic subvariety.
		\item A (weakly) special subvariety of $U$ is a complex-analytically irreducible component of $\pi^{-1}(W)$, where $W$ is a (weakly) special subvariety of $X$.
	\end{enumerate}
\end{definition}

\begin{remark}\label{rmk:wkspec}
	The weakly special subvarieties of $Y(1)^m$ are, equivalently, those subvarieties of $Y(1)^m$ which may be defined by equations of the form $\Phi_N(x_i, x_k) = 0$ and $x_l = c$ for $c \in \C$ constant. Here $\Phi_N$ denotes the $N$th classical modular polynomial. For a proof of this equivalence, see \cite[\S2]{BiluLucaMasser17}.
\end{remark}

The Ax--Schanuel result we need is due to Pila and Tsimerman \cite{PilaTsimerman17}.

\begin{thm}[{\cite[Theorem~3.2]{PilaTsimerman17}}]\label{thm:AxS}
	Let $V \subset X$ and $W \subset U$ be algebraic subvarieties and $A \subset W \cap \pi^{-1}(V)$ a complex-analytically irreducible component. Then
	\[\dim A=\dim V + \dim W - \dim X,\]
	unless $A$ is contained in a proper weakly special subvariety of $U$.
\end{thm}

This is an Ax--Schanuel result of ``mixed'' type and follows from the corresponding results in the two extreme cases: for the functions $\exp \colon \C^n \to \G^n$ and  $j \colon \h^m \to Y(1)^m$. (Recall that, by an abuse of notation, we write $\exp$ and $j$ for arbitrary Cartesian products of the maps $\exp \colon \C \to \G$, $j \colon \h \to \C$.) For $\exp \colon \C^n \to \G^n$, the necessary result follows from Ax's theorem \cite{Ax71}. For the case of $j \colon \h^m \to Y(1)^m$, the result is proved in \cite{PilaTsimerman16}.

From Theorem~\ref{thm:AxS}, one may deduce the following Ax--Schanuel statement.

\begin{thm}[{\cite[Theorem~3.3]{PilaTsimerman17}}]\label{thm:AxS1}
	Let $U' \subset U$ be a weakly special subvariety and $X' = \pi(U')$. Let $V \subset X', W \subset U'$ be algebraic subvarieties, and let $A \subset W \cap \pi^{-1}(V)$ be a complex-analytically irreducible component. Then
	\[\dim A=\dim V + \dim W - \dim X',\]
	unless $A$ is contained in a proper weakly special subvariety of $U'$.
\end{thm}

This statement may itself be equivalently formulated in terms of so-called optimal varieties, see \cite[Section~5]{HabeggerPila16}. We need the following definitions from \cite{HabeggerPila16}.

\begin{definition}
Fix a subvariety $V \subset X$.
	\begin{enumerate}
		\item A component (with respect to $V$) is a complex-analytically irreducible component of $W \cap \pi^{-1}(V)$ for some algebraic subvariety $W \subset U$.
		\item The defect $\defect (A)$ of such a component $A$ is $\defect (A) = \dim \mathrm{Zcl}(A) - \dim A$, where $\mathrm{Zcl}(A)$ denotes the Zariski-closure of $A$.
		\item A component $A$ is called optimal if there is no strictly larger component $A' \supsetneq A$ with $\defect (A') \leq \defect(A)$.
		\item A component $A$ is called geodesic if the algebraic subvariety $W = \mathrm{Zcl}(A) \subset U$ in the definition of a component may be taken to be a weakly special subvariety. 
	\end{enumerate}
\end{definition}

Theorem~\ref{thm:AxS1} is then equivalent (\cite[Section~5]{HabeggerPila16}) to the following.

\begin{thm}[{\cite[Proposition~3.5]{PilaTsimerman17}}]\label{thm:AxSopt}
	Let $V \subset X$ be a subvariety. Let $A$ be a component with respect to $V$. If $A$ is optimal, then $A$ is geodesic.
\end{thm}

We conclude this section with the definition of an atypical component of a subvariety $V$ of the mixed Shimura variety $X$. We will need this definition in Sections~\ref{sec:ZP} and \ref{sec:ZPgam} when we relate our results to the Zilber--Pink conjecture.

\begin{definition}\label{def:atyp}
Let $V \subset X$ be a subvariety. A subvariety $A \subset V$ is called an atypical component of $V$ in $X$ if there is a special subvariety $T \subset X$ such that $A \subset V \cap T$ and
	\[\dim A > \dim V + \dim T - \dim X.\]
\end{definition} 

\begin{remark}\label{rmk:atyp}
	Usually the ambient mixed Shimura variety $X$ will be clear from context, and so we will just refer to atypical components of $V$.
\end{remark}

\section{Finiteness of multiplicatively dependent tuples}\label{sec:tup}
 
We now prove Theorem~\ref{thm:moddep}. Our proof will be in two stages. First we will prove the theorem for the case of modular functions $f \in \alg(j)$ via an o-minimal counting argument. We then extend the result to modular functions $f \in \C(j)$.

For the first step, we prove the conditional result Proposition~\ref{prop:conjtup}, which covers multiplicatively dependent $f$-special points for a modular function $f \in \alg(j)$ (which does not necessarily satisfy the divisor condition). This result is conditional on a functional independence statement for distinct $\GL$-translates of $f$, which we formulate below as Condition~\ref{conj:ind}. Theorem~\ref{thm:modind} establishes this statement for modular functions satisfying the divisor condition. We may therefore deduce Proposition~\ref{prop:divconjtup}, which is Theorem~\ref{thm:moddep} for functions $f \in \alg(j)$.

We adopt this approach in order to emphasise that, for the case of $f \in \alg(j)$, our proof of Theorem~\ref{thm:moddep} does not depend on any particular facts about the divisor condition. As we mentioned in Remark~\ref{rmk:necessary}, it may be possible to establish Condition~\ref{conj:ind} under a weaker hypothesis than the divisor condition. If one could do this, then Proposition~\ref{prop:conjtup} shows that, for $f \in \alg(j)$, Theorem~\ref{thm:moddep} would hold under this weaker hypothesis.

\begin{cond}\label{conj:ind}
	Let $f \colon \h \to \C$ be a non-constant modular function. If the functions $f(g_1z), \ldots, f(g_nz)$ are pairwise distinct for $g_1, \ldots, g_n \in \GL$, then they are multiplicatively independent modulo constants.
\end{cond}

\begin{prop}\label{prop:conjtup}
	Let $f \colon \h \to \C$ be a non-constant modular function such that $f \in \alg(j)$. Suppose that $f$ satisfies Condition~\ref{conj:ind}. Let $n \geq 1$. Then there exist only finitely many $n$-tuples $(\sigma_1, \ldots, \sigma_n)$ of distinct $f$-special points such that the set $\{\sigma_1, \ldots, \sigma_n \}$ is multiplicatively dependent, but no proper subset of $\{\sigma_1, \ldots, \sigma_n \}$ is multiplicatively dependent.
\end{prop}

\begin{prop}\label{prop:divconjtup}
	Let $f \colon \h \to \C$ be a non-constant modular function such that $f \in \alg(j)$. Suppose that $f$ satisfies the divisor condition. Let $n \geq 1$. Then there exist only finitely many $n$-tuples $(\sigma_1, \ldots, \sigma_n)$ of distinct $f$-special points such that the set $\{\sigma_1, \ldots, \sigma_n \}$ is multiplicatively dependent, but no proper subset of $\{\sigma_1, \ldots, \sigma_n \}$ is multiplicatively dependent.
\end{prop}

\begin{proof}[Proof of Proposition~\ref{prop:divconjtup}]
	Let $f$ be a non-constant modular function satisfying the divisor condition. Then Condition~\ref{conj:ind} holds for $f$ by Theorem~\ref{thm:modind}. Proposition~\ref{prop:divconjtup} thus follows from Proposition~\ref{prop:conjtup}.
	\end{proof}

Now we come to the proof of Proposition~\ref{prop:conjtup}. We collect first those arithmetic estimates on $f$-special points which we will require.

\subsection{Arithmetic estimates}\label{subsec:arith}
Let $f \colon \h \to \C$ be a non-constant modular function such that $f \in \alg(j)$. Then $f(z) = R(j(z))$ for some rational function $R$ with algebraic coefficients. Let $\sigma$ be an $f$-special point. So $\sigma = f(\tau)$ for some $\tau \in \h$ with $[\Q(\tau) : \Q] =2$. Observe that the function $f$ is $\SL$-invariant, the restricted function $j|_{F_j} \colon F_j \to \C$ is injective, and the (non-constant) rational map $R \colon \C \to \C$ is finite-to-one on its domain of definition. Thus there exist $\tau_1, \ldots, \tau_k \in F_j$, with $k \geq 1$ and $[\Q(\tau_i) : \Q]=2$ for every $i$, with the property that, for every $\tau \in F_j$ with $[\Q(\tau) : \Q]=2$, one has $f(\tau) = \sigma$ if and only if $\tau \in \{\tau_1, \ldots, \tau_k\}$. The discriminant $\Delta(\tau_i)$ of the quadratic number $\tau_i$ is defined as $\Delta(\tau_i) = b^2-4ac$, where $a,b,c \in \Z$ are such that $a\tau_i^2 +b \tau_i + c=0$ and $\gcd(a,b,c)=1$. (Note in particular that $\Delta(\tau_i)<0$ for every $\tau_i$.) We then define the discriminant $\Delta(\sigma)$ of the $f$-special point $\sigma$ by $\Delta(\sigma) = \min \{\Delta(\tau_1), \ldots, \Delta(\tau_k)\}$. Observe that the discriminant of the $j$-special point $j(\tau_i)$ is equal to $\Delta(\tau_i)$.

 Write $H(\alpha)$ for the absolute Weil height of an algebraic number $\alpha$, and $h(\alpha) = \log H(\alpha)$ for the absolute logarithmic Weil height. For $\alpha = (\alpha_1, \ldots, \alpha_n) \in \alg^n$, we define
  \[H(\alpha) = \max \{H(\alpha_i) : i=1, \ldots, n\}\]
  and
 \[h(\alpha) = \max \{h(\alpha_i) : i=1, \ldots, n\}.\] 
 In the remainder of this paper, references to ``the height'' of $\alpha$ will mean $H(\alpha)$; we will refer to $h(\alpha)$ as the logarithmic height. In this subsection, constants will be positive and with only the indicated dependencies, but in general not effective. Let $K$ be a number field containing the coefficients of $R$. The arithmetic estimates we need are the following.

\begin{lem}\label{lem:htspec}
	For all $\epsilon>0$, there exists a constant $c(\epsilon, f)$ such that
$h(\sigma) \leq c(\epsilon, f) \lvert \Delta \rvert^\epsilon$ for every $f$-special point $\sigma$ of discriminant $\Delta$.
\end{lem}

\begin{proof}
	Fix $\epsilon>0$. Let $\tau_0 \in F_j$ be such that $\sigma = f(\tau_0)$ and $[\Q(\tau_0) : \Q]=2$. Note that $\lvert \Delta(\tau_0) \rvert \leq \lvert \Delta \rvert$. There exists \cite[(5.4)]{PilaTsimerman17} a constant $c(\epsilon)$ such that
	\[h(j(\tau_0))\leq c(\epsilon) \lvert \Delta(\tau_0) \rvert^\epsilon.\]
	Recall that $\sigma = R(j(\tau_0))$, where $R$ is a rational function defined over $K$. The result then follows thanks to elementary properties of the logarithmic height: 
	\[h(xy) \leq h(x) + h(y) \mbox{ for } x, y \in \alg,\]
	\[ h(x_1 + \ldots + x_r) \leq h(x_1) + \ldots + h(x_r) + \log r \mbox{ for } x_1 \ldots, x_r \in \alg,\]
	\[ h(x^r)= \lvert r \rvert h(x) \mbox{ for } r \in \Q, x \in {\alg}^{\times},\]
	see e.g. \cite[Chapter~1]{BombieriGubler06}.
	\end{proof}

\begin{lem}\label{lem:htpre}
		 Let $\sigma$ be an $f$-special point of discriminant $\Delta$ and suppose $\sigma = f(\tau_0)$ for $\tau_0 \in F_j$ such that $[\Q(\tau_0) : \Q]=2$. Then $H(\tau_0) \leq 2 \lvert \Delta \rvert$.
\end{lem}

\begin{proof}
	This follows immediately from \cite[(5.5)]{PilaTsimerman17}.
	\end{proof}

\begin{lem}\label{lem:degspec}
	For any $\epsilon >0$ there exist constants $c_1(\epsilon, f), c_2(\epsilon, f)$ such that 
	\[c_1(\epsilon, f)\lvert \Delta \rvert^{\frac{1}{2}-\epsilon} \leq [\Q(\sigma) : \Q] \leq c_2(\epsilon, f)\lvert \Delta \rvert^{\frac{1}{2}+\epsilon}\]
	for all $f$-special points $\sigma$ of discriminant $\Delta$. Hence, there are also constants $c_1'(\epsilon, f), c_2'(\epsilon, f)$ such that
	\[c_1'(\epsilon, f)\lvert \Delta \rvert^{\frac{1}{2}-\epsilon} \leq [K(\sigma) : K] \leq c_2'(\epsilon, f)\lvert \Delta \rvert^{\frac{1}{2}+\epsilon}.\]
	\end{lem}

\begin{proof}
	Fix $\epsilon>0$. There exists $c(\epsilon)$ such that 
	\[c(\epsilon)\lvert \Delta(\tau) \rvert^{\frac{1}{2}-\epsilon} \leq [\Q(j(\tau)) : \Q] \leq c(\epsilon)\lvert \Delta(\tau) \rvert^{\frac{1}{2}+\epsilon},\]
	for every $\tau \in \h$ with $[\Q(\tau) : \Q]=2$, see \cite[(5.6), (5.7)]{PilaTsimerman17}. We will show that there exists a constant $M(f)>0$ such that
	\begin{equation}\label{eqn:ineq}
		\frac{1}{M(f)} [\Q(j(\tau)) : \Q] \leq [\Q(f(\tau)) : \Q] \leq M(f) [\Q(j(\tau)) : \Q]
		\end{equation}
	holds for all $\tau \in \h$ with $[ \Q(\tau) : \Q] =2$. The first part of the lemma then follows by combining these two inequalities, each applied to some $\tau_1 \in F_j$ with the property that $\sigma = f(\tau_1)$, $[\Q(\tau_1) : \Q]=2$,  and $\Delta(\tau_1) =\Delta$. 
	
	For $f \in \Q(j)$ non-constant, the necessary inequality~\eqref{eqn:ineq} was proved by Spence in \cite{Spence17} (which treats actually the broader case $f \in \Q(j , \chi^*)$, where $\chi^*$ is a certain almost holomorphic modular function). The extension to $f \in \alg(j)$ was given by Spence in private communication with the author; we include a proof below.
	
	Let $d = [ K : \Q]$ and take $\tau \in \h$ with $[ \Q(\tau) : \Q] =2$. For the upper bound, observe that
	\[[\Q(f(\tau)) : \Q] \leq [K(j(\tau)) : \Q] = [K(j(\tau)) : \Q(j(\tau))][\Q(j(\tau)) : \Q] \leq d [\Q(j(\tau)) : \Q].\]
	
	Now for the lower bound. Suppose that $[\Q(j(\tau), f(\tau)) : \Q(f(\tau))]>M$ for some $M$. Then
	\begin{align*}
		&[K(j(\tau), f(\tau)) : K(f(\tau))][K(f(\tau)) : \Q(f(\tau))] \\
		= &[K(j(\tau), f(\tau)) : \Q(f(\tau))]\\
		= &[K(j(\tau), f(\tau)) : \Q(j(\tau), f(\tau))] [\Q(j(\tau), f(\tau)) : \Q(f(\tau))]\\
		> &[K(j(\tau), f(\tau)) : \Q(j(\tau), f(\tau))] \cdot M.
	\end{align*}
	Hence, using that $K(j(\tau))=K(j(\tau), f(\tau))$ since $f(\tau)$ is $K$-rational in $j(\tau)$, we have that
\[[K(j(\tau)) : K(f(\tau))] >  \frac{[K(j(\tau), f(\tau)) : \Q(j(\tau), f(\tau))]}{[K(f(\tau)) : \Q(f(\tau))]} \cdot M \geq  \frac{M}{d} .\]
	
	Taking Galois conjugates of $j(\tau)$ over $K(f(\tau))$, we then obtain at least $M/d$ distinct quadratic points $\tau'$ in the standard fundamental domain $F_j$ with $f(\tau')=f(\tau)$. For suitably large $M$ though, this is impossible since $f$ is meromorphic and, restricted to $F_j$, definable in the o-minimal structure $\mathbb{R}_{\mathrm{an}, \exp}$, so uniform finiteness applies (see e.g. \cite[Chapter~3]{vandenDries98}).
	
	For $M$ large enough (depending on $f$), we therefore have that $[\Q(j(\tau), f(\tau)) : \Q(f(\tau))] \leq M$. Fix such an $M$. Then $[\Q(j(\tau), f(\tau)) : \Q] \leq M [\Q(f(\tau)) : \Q]$, and so
	\[[\Q(f(\tau)) : \Q] \geq \frac{1}{M} [\Q(j(\tau), f(\tau)) : \Q] \geq \frac{1}{M} [\Q(j(\tau)) : \Q].\]
	Therefore, taking $M(f)=\max\{M, d\}$ completes the proof of \eqref{eqn:ineq}.
	
	For the second part of the lemma, the upper bound follows since
	\[[K(\sigma) : K] \leq [\Q(\sigma) : \Q],\]
	while the lower bound comes from the fact that
		\[[K(\sigma) : K] [K : \Q] = [K(\sigma) : \Q] \geq [\Q(\sigma) : \Q],\]
		and thus
		\[[K(\sigma) : K] \geq \frac{1}{d}[\Q(\sigma) : \Q]. \qedhere\]	
	\end{proof}

\begin{lem}\label{lem:htmultdep}
	Let $\alpha_1, \ldots, \alpha_n$ be multiplicatively dependent non-zero elements of a number field $L$ of degree $d \geq 2$. Suppose that any proper subset of the $\alpha_i$ is multiplicatively independent. Then there exists a constant $c(n)$ and $b_1, \ldots, b_n \in \Z$, not all zero, such that $\alpha_1^{b_1} \cdots \alpha_n^{b_n}=1$ and 
\[	\vert b_i \vert \leq c(n) d^n \log d \prod_{\substack{j=1\\ j \neq i}}^{n} h(\alpha_j), \quad i=1, \ldots, n.\]
\end{lem}

\begin{proof}
	See \cite[Corollary~3.2]{LoherMasser04} for the case $n \geq 2$. If $n=1$, then $\alpha_1$ is a root of unity of degree $\leq d$, and hence the desired result follows from elementary bounds for the Euler $\phi$-function. 
	\end{proof}

\subsection{Conjugates of $f$-special points}\label{subsec:conj}

Write $T_{\Delta}$ for the set of integer triples $(a,b,c)$ such that: $\mathrm{gcd}(a,b,c)=1$, $\Delta = b^2-4ac$, and either $-a < b \leq a < c$ or $0 \leq b \leq a = c$. Then there is a bijection between $T_{\Delta}$ and the $j$-special points of discriminant $\Delta$ given by $(a,b,c) \mapsto j((b + \sqrt{\Delta})/2a)$. The $j$-special points of discriminant $\Delta$ form a full Galois orbit over $\Q$. See for example \cite[\S2.2]{BiluLucaMadariaga16}.

\begin{lem}\label{lem:conj}
	Let $K$ be a number field, $R \in K(t)$ a non-constant rational function, and $f \colon \h \to \C$ the modular function defined by $f(z) = R(j(z))$. Let $\sigma$ be an $f$-special point of discriminant $\Delta$. Then the set 
	 \[\Big\{ f\Big(\frac{b + \sqrt{\Delta}}{2a}\Big)\colon (a, b, c) \in T_\Delta \Big\}\]
	 contains all $K$-conjugates of $\sigma$, which are also $f$-special points of discriminant $\Delta$.
\end{lem}

\begin{proof}
	 We may write $\sigma = R(j)$, where $j$ is a $j$-special point of discriminant $\Delta$. Let $\sigma'$ be a $K$-conjugate of $\sigma$. Since $R$ is a rational function over $K$, one has that $\sigma' = R(j')$ where $j'$ is a $K$-conjugate of $j$. But we have a complete description of the conjugates of $j$, they arise as the $j$-special points $j((b + \sqrt{\Delta})/2a)$ for $(a,b,c) \in T_\Delta$. So $\sigma'$ is $f$-special and belongs to the set in the lemma.
	 
	 It remains to show that the discriminant of the $f$-special point $\sigma'$ is equal to $\Delta$. Since $\sigma'=R(j')$ and $j'$ is a $j$-special point of discriminant $\Delta$, this will be true unless $\sigma' = R(j_0)$ for some $j$-special point $j_0$ of discriminant $\Delta_0 < \Delta$. If this were to happen, then one would have that $\sigma = R(j_0')$ for some $K$-conjugate $j_0'$ of $j_0$ since $\sigma, \sigma'$ are $K$-conjugate. But $j_0'$ would be a $j$-special point of discriminant $\Delta_0$, and so the discriminant of $\sigma$ would have to be $\leq \Delta_0$, implying that $\Delta \leq \Delta_0 < \Delta$, a contradiction.
\end{proof}

The proof of Proposition~\ref{prop:conjtup} will be by an o-minimal counting argument. We first establish a lower bound for the number of multiplicatively dependent tuples of $f$-special points of a given discriminant. We will then apply the Pila--Wilkie o-minimal Counting Theorem \cite{PilaWilkie06} (and in particular its extension to algebraic points \cite{Pila09}) to count the preimages of such tuples. The lower bound we use comes from a lower bound for the size of the Galois orbit of an $f$-special point. We will thus require that conjugates of $f$-special points are again $f$-special. This is why the o-minimal argument in Subsection~\ref{subsec:conjtup} treats only the case of $f \in \alg(j)$.

\subsection{Proof of Proposition~\ref{prop:conjtup}}\label{subsec:conjtup}

Let $f \colon \h \to \C$ be a non-constant modular function such that Condition~\ref{conj:ind} holds for $f$. Suppose also that $f \in \alg(j)$, so that $f(z) = R(j(z))$ for some rational function $R \in \alg(t)$. Let $K$ be a number field containing the coefficients of $R$. Fix $n \geq 1$. Let $X = X_{n,n}= Y(1)^n \times \G^n$ and
\[V = \{(x_1, \ldots, x_n, t_1, \ldots, t_n) \in X : t_i = R(x_i) \mbox{ for } i=1, \ldots, n\} \subset X.\]
Then $\dim V = \codim V = n$. Recall the definition of a (weakly) special subvariety of $X$ from Section~\ref{sec:ax}. The proof of Proposition~\ref{prop:conjtup} is similar to the proof of \cite[Theorem~1.2]{PilaTsimerman17}.

\begin{proof}[Proof of Proposition~\ref{prop:conjtup}]
In the following, constants $c, c'$ are positive and depend only on our choice of $f$ and $n$, but will vary between occurrences.
	
	The complexity $\Delta(\sigma)$ of an $n$-tuple $\sigma = (\sigma_1, \ldots, \sigma_n)$ of $f$-special points is defined to be $\max \{\lvert \Delta(\sigma_1) \rvert, \ldots, \lvert \Delta(\sigma_n) \rvert \}$. We define an $f$-dependent tuple to be an $n$-tuple $\sigma = (\sigma_1, \ldots, \sigma_n)$ of distinct $f$-special points satisfying a non-trivial multiplicative relation and minimal for this property.
	
	Let $F_j$ denote the standard fundamental domain for the action of $\SL$ on $\h$, which we defined in Section~\ref{sec:intro}, and let $F_{\exp}$ denote the standard fundamental domain for the action of $2 \pi i \Z$ on $\C$ given by
	$F_{\exp}=\{ z \in \C : 0 \leq \im z < 2 \pi\}$. Let
	\[Y = \Big \{(z,u,r,s) \in F_j^n \times F_{\exp}^n \times \R^n \times \R : R(j(z))=\exp(u), r \cdot u = 2 \pi i s\Big\}\]
	and
	\[Z = \Big \{(z,r,s) \in F_j^n \times \R^n \times \R : \exists u (z,u,r,s) \in Y\}.\]
	Then $(j(z), \exp(u)) \in V$ for $(z,u,r,s) \in Y$, and the sets $Y, Z$ are definable in the o-minimal structure $\R_{\mathrm{an}, \exp}$.
	
	An $f$-dependent tuple $\sigma = (\sigma_1, \ldots, \sigma_n)$ of complexity $\Delta$ gives rise to a point $(x_1, \ldots, x_n, \sigma_1, \ldots, \sigma_n) \in V$, where each $x_i$ is a $j$-special point of discriminant $\Delta(\sigma_i)$ satisfying $R(x_i) = \sigma_i$. This point in $V$ has a preimage 
	\[\tau = (z_1, \ldots, z_n, u_1, \ldots, u_n) \in F_j^n \times F_{\exp}^n.\] Now $\tau$ gives rise to the point 
	\[(z_1, \ldots, z_n, u_1, \ldots, u_n, b_1, \ldots, b_n, b) \in Y,\]
 where the $b_i, b$ are rational integers recording the multiplicative dependence of $\sigma_1, \ldots, \sigma_n$, i.e. such that 
	\[\sum_{i=1}^{n} b_i u_i = 2 \pi i b.\]
	
	Since $x_i$ is a $j$-special point of discriminant $\Delta(\sigma_i)$, so $z_i$ is a quadratic point with height bounded by $2 \lvert \Delta(\sigma_i) \rvert$ by Lemma~\ref{lem:htpre}. The $\sigma_i$ are $f$-special points of discriminants $\Delta(\sigma_i)$ such that the set $\{\sigma_1, \ldots, \sigma_n\}$ is multiplicatively dependent and minimal for this property. We may therefore use the bounds on the logarithmic height (Lemma~\ref{lem:htspec}) and degree (Lemma~\ref{lem:degspec}) of the $\sigma_i$, together with the result in Lemma~\ref{lem:htmultdep}, to see that the integers $b_i$ may be chosen to have absolute value bounded by $c \Delta^{n^2}$. Since 
	\[\sum_{i=1}^n b_i u_i = 2 \pi i b\]
	and $0 \leq \im (u_i) < 2 \pi$, we obtain also a bound on $\lvert b \rvert$. Thus, the height of the point
	\[(z_1, \ldots, z_n, b_1, \ldots, b_n, b) \in Z\]
	is bounded by $c \Delta^{n^2}$ since $H(k) = \lvert k \rvert$ for $k \in \Z \setminus \{0\}$. Further, this point is quadratic in the $z_i$ coordinates and rational (even integral) in the $b_i, b$ coordinates.
	
	Suppose then that there are infinitely many $f$-dependent tuples. Then, in particular, there are $f$-dependent tuples of arbitrarily large complexity $\Delta$. Let $\sigma$ be an $f$-dependent tuple of suitably large complexity $\Delta$. By Lemma~\ref{lem:degspec}, $\sigma$ has at least $c \Delta^{1/4}$ conjugates over $K$, and by Lemma~\ref{lem:conj} each of these conjugates $\sigma'$ is itself an $f$-dependent tuple of complexity $\Delta$. Further, all these conjugate tuples satisfy the same multiplicative relation. Consequently, each of these $\sigma'$ gives rise to a point
		\[(z_1', \ldots, z_n', u_1', \ldots, u_n', b_1, \ldots, b_n, b') \in Y,\]
	which is quadratic in the $z_i'$ coordinates and rational (even integral) in the $b_i, b'$ coordinates. Note that the coordinates $(b_1, \ldots, b_n)$ are the same for all these conjugates, while the coordinates $(z_1', \ldots, z_n')$ and $(u_1', \ldots, u_n')$ must be distinct for distinct conjugates. Further, the height of each of the corresponding points
	\[(z_1', \ldots, z_n', b_1, \ldots, b_n, b') \in Z\]
	is bounded by $c \Delta^{n^2}$.
	
	We are now in a position to apply the o-minimal Counting Theorem in the form of \cite[Corollary~7.2]{HabeggerPila16}. View $Y$ as a definable family parametrised by $\R^n$. Let $Y_{(b_1, \ldots, b_n)}$ denote the fibre of $Y$ over the point $(b_1, \ldots, b_n) \in \R^n$. Let $\Sigma \subset Y_{(b_1, \ldots, b_n)}$ denote the set of points of $Y_{(b_1, \ldots, b_n)}$ arising from the conjugates of $\sigma$. Let $\pi_1, \pi_2$ denote the projections of $Y_{(b_1, \ldots, b_n)}$ to $F_j^n \times \R$ and $F_{\exp}^n$ respectively. Provided that $\Delta$ is suitably large (which we may always assume), we thereby obtain a continuous, definable function $\beta \colon [0, 1] \to Y_{(b_1, \ldots, b_n)}$ such that:
	\begin{enumerate}
		\item The composition $\pi_1 \circ \beta \colon [0, 1] \to F_j^n \times \R$ is semialgebraic and its restriction to $(0,1)$ is real analytic.
		\item The composition $\pi_2 \circ \beta \colon [0, 1] \to F_{\exp}^n$ is non-constant.
		\item $\beta(0) \in \Sigma$.
	\end{enumerate}

(In \cite[Corollary~7.2]{HabeggerPila16}, it is only stated that $\pi_2(\beta(0)) \in \pi_2(\Sigma)$, but in fact the authors prove there that $\beta(0) \in \Sigma$.) Note also that (2) implies that the composition of $\beta$ with projection to the $F_j^n$ coordinates is non-constant. Denote by $Z_{(b_1, \ldots, b_n)}$ the fibre of $Z$ over the point $(b_1, \ldots, b_n)$. Note that $(\pi_1 \circ \beta)([0, 1] ) \subset Z_{(b_1, \ldots, b_n)}$. There thus exists a continuous, semialgebraic map $\gamma \colon [0, 1] \to Z_{(b_1, \ldots, b_n)}$ which projects non-constantly to the $F_j^n$ coordinates, maps $0$ to a point of $Z_{(b_1, \ldots, b_n)}$ corresponding to a conjugate of $\sigma$, and whose restriction to $(0, 1)$ is real analytic. 

Let $\Phi \colon Z_{(b_1, \ldots, b_n)} \to \h^n \times \C$ be given by $(z, s) \mapsto (z, 2 \pi i s)$. Observe that the set $(\Phi \circ \gamma)([0, 1])$ is contained in the set
\[ \tilde{V} = \{(z, t) \in \h^n \times \C : \prod_{i=1}^n f(z_i)^{b_i} = \exp(t)\}.\]
There then exists an open neighbourhood $\Omega \subset \h^n \times \C$ of $\Phi(\gamma(0))$ and a set $P$ such that:
\begin{enumerate}
	\item $\Omega \cap (\Phi \circ \gamma)([0, 1]) \subset P \subset \tilde{V}$
	\item $P$ may be written as a finite union of irreducible Nash subsets of $\Omega$ which each contain $\Phi(\gamma(0))$.
\end{enumerate}
 
This follows from \cite[Proposition~1]{AdamusRand13} and the characterisation of Nash sets given in \cite[p.~989]{AdamusRand13}.

If every complex-analytically irreducible component of $P$ had constant projection to its $\h^n$ coordinates, then $(\Phi \circ \gamma)([0, 1])$ would have constant projection to its $\h^n$ coordinates by real analytic continuation. But $(\Phi \circ \gamma)([0, 1])$ does not have constant projection to its $\h^n$ coordinates. Hence, there must exist some complex-analytically irreducible component of $P$ which has non-constant projection to its $\h^n$ coordinates. Observe that every complex-analytically irreducible component of $P$ contains $\Phi(\gamma(0))$. Hence, by complex analytic continuation, there exists a complex algebraic subvariety $W \subset \C^{n+1}$ and a complex-analytically irreducible component $A \subset (\h^n \times \C) \cap W$ such that $\Phi(\gamma(0)) \in A \subset \tilde{V}$ and $A$ has non-constant projection to its $\h^n$ coordinates.

By the Ax--Schanuel results of Section~\ref{sec:ax}, there exist weakly special subvarieties $W_1 \subset \h^n$ and $W_2 \subset \C$ such that $A \subset W_1 \times W_2 \subset \tilde{V}$. The projection $\tilde{V} \to \h^n$ has discrete fibres, and hence $W_2$ must just be a point, which therefore has to be equal to the projection of $\Phi(\gamma(0))$. So $W_2 = \{2 \pi i b'\}$ for some $b' \in \Z$.
	
	Therefore, $W_1$ is a positive-dimensional weakly special subvariety of $\h^n$ which is contained in the set 
	\[\{ z \in \h^n : \prod_{i=1}^n f(z_i)^{b_i} = 1\}.\] 
	In addition, $W_1$ contains a pre-image $(\tau_1, \ldots, \tau_n) \in \h^n$ of some $f$-dependent tuple $(f(\tau_1), \ldots, f(\tau_n))$. In particular, $f(W_1)$ has no two identically equal coordinates, since the $f(\tau_i)$ are pairwise distinct. Note that all the $b_i$ are non-zero, since the set $\{f(\tau_1), \ldots, f(\tau_n)\}$ is minimally multiplicatively dependent. Taking the image of $W_1$ under $f$, we therefore obtain a multiplicative dependence modulo constants among some pairwise distinct $\GL$-translates of $f$. This contradicts Condition~\ref{conj:ind}, and so we are done.
	\end{proof}

\subsection{Proof of Theorem~\ref{thm:moddep}}\label{subsec:complextuples}

Theorem~\ref{thm:moddep} follows from Proposition~\ref{prop:divconjtup} by a specialisation argument. We will need the following elementary lemma.

\begin{lem}\label{lem:minmult}
	Let $n \geq 1$. Suppose that $x_1, \ldots, x_n \in \G$ are pairwise distinct and such that
	\[x_1^{a_1} \cdots x_n^{a_n} = 1,\]
	for some $a_i \in \Z$ with $a_1 \neq 0$. Then there exists a set $S \subset \{x_1, \ldots, x_n\}$ such that $x_1 \in S$ and $S$ is minimally multiplicatively dependent.
\end{lem}

\begin{proof}
	We proceed by induction on $n$. If $n=1$, then the result is trivial. Now let $n > 1$. Suppose that $x_1, \ldots, x_n \in \G$ are pairwise distinct and such that
	\[x_1^{a_1} \cdots x_n^{a_n} = 1\]
	for some $a_i \in \Z$ with $a_1 \neq 0$. Suppose, for contradiction, that there is no minimally multiplicatively dependent subset of $\{x_1, \ldots, x_n\}$ which contains $x_1$. Then, in particular, the set $\{x_2, \ldots, x_n\}$ must be multiplicatively dependent. So
	\[ x_2^{b_2} \cdots x_n^{b_n} = 1\] 
	for some $b_2, \ldots, b_n \in \Z$ not all zero. Without loss of generality, we assume that $b_2 \neq 0$. We may then eliminate $x_2$ from the first equality, to obtain that
	\[ x_1^{a_1 b_2} x_3^{a_3 b_2 - a_2 b_3} \cdots x_n^{a_n b_2 - a_2 b_n} = 1.\]
	Note that $a_1 b_2 \neq 0$. Thus, by induction, there exists a set $S \subset \{x_1, x_3, \ldots, x_n\}$ such that $x_1 \in S$ and $S$ is minimally multiplicatively dependent. This gives the desired contradiction.
	\end{proof}

Now we may proceed with the proof of Theorem~\ref{thm:moddep}. For the remainder of this section, fix $f \in \C(j)$ to be a non-constant modular function satisfying the divisor condition. Write $f(z) = R(j(z))$ for some rational function $R(t) \in \C(t)$. We may assume that $R(t) \notin \alg(t)$, else we would be in the case of Proposition~\ref{prop:divconjtup}. There are $N \geq 1$ and $a = (a_0, \ldots, a_{N-1}) \in \C^N$ such that
\[R(t) = \frac{p(a, t)}{q(a,t)},\]
where
\[ p(a, t) = a_0 \prod_{i=1}^k (t - a_i)\]
and
\[ q(a, t) = \prod_{i=k+1}^{N-1} (t - a_i).\]
Note that at least one coordinate $a_i$ of $a$ must be transcendental since we have assumed that $R(t) \notin \alg(t)$. We may also assume that the polynomials $p(a, t), q(a, t) \in \C[t]$ are relatively prime in $\C[t]$.

For $a' \in \C^N$, define $f_{a'} \colon \h \to \C$ by
\[f_{a'}(z) = \frac{p(a', j(z))}{q(a', j(z))}.\]
So, in particular, $f =f_a$. Note that, for every $a' \in \C^N$, the function $f_{a'}$ is a modular function. Clearly, if $a' \in \alg^N$, then $f_{a'} \in \alg(j)$. We need the following fact about the divisor condition.

\begin{lem}\label{lem:open}
	Let $W \subset \C^N$ be the smallest affine variety over $\alg$ such that $a \in W$. Suppose the divisor condition holds for $f$. Then there exists an open neighbourhood $U$ of $a$ such that, for every $a' \in W \cap U$, the divisor condition holds for the modular function $f_{a'}$.
\end{lem}

\begin{proof}
	Let $U$ be an open neighbourhood of $a$ and take $a' = (a_0', \ldots, a_{N-1}') \in U$. Provided $U$ is chosen suitably small, the divisor condition will hold for the modular function $f_{a'}(z)$, unless possibly one of the following happens:
	\begin{enumerate}
		\item $a_i = a_k$, but $a_i' \neq a_k'$ for some $i, k$;
		\item $a_i \in \{0, 1728\}$ and $a_i' \neq a_i$ for some $i$.
	\end{enumerate}
Both these possibilities are ruled out though if $a' \in W$.
	\end{proof}

We now complete the proof of Theorem~\ref{thm:moddep}.

\begin{proof}[Proof of Theorem~\ref{thm:moddep}]
 Fix $n \geq 1$. Let $Y$ be the set of $j$-special points. Let $Z$ be the set of $y = (y_1, \ldots, y_n) \in Y^n$ such that the $f$-special points $R(y_1), \ldots, R(y_n)$ are pairwise distinct and the set $\{R(y_1), \ldots, R(y_n)\}$ is minimally multiplicatively dependent. We need to show that the set $Z$ is finite. Suppose then, for contradiction, that $Z$ is infinite. Then, in particular, the set
	\[ \tilde{Z} = \{ \sigma \in Y : \exists y \in Z \mbox{ such that } \sigma = \pi_i(y) \mbox{ for some } i=1, \ldots, n \}\]
	is infinite, where $\pi_i \colon Y^n \to Y$ denotes the projection map to the $i$th coordinate. 
	
	If $y = (y_1, \ldots, y_n)\in Z$, then there are $m_1, \ldots, m_n \in \Z \setminus \{0\}$ such that
	\[p(a, y_1)^{m_1} \cdots p(a, y_n)^{m_n} = q(a, y_1)^{m_1} \cdots q(a, y_n)^{m_n}. \]
	Multiplying by those factors with $m_i < 0$, we get a corresponding polynomial equality, which we denote $(*)$. (Note that the equality ($*$) depends on $y$, but is always of the same form.) In addition, one has that
	\begin{equation}\label{eq:nonzero}
		p(a, y_i) \neq 0 \mbox{ and } q(a, y_i) \neq 0 \mbox{ for } i=1, \ldots, n\tag{$**$}
		\end{equation}
	for all $y \in Z$. Since $R(y_1), \ldots, R(y_n)$ are pairwise distinct, one also has that
	\begin{equation}\label{eq:distinct}
		p(a, y_i) q(a, y_k) - p(a, y_k) q(a, y_i) \neq 0\tag{$***$}
		\end{equation}
	for all $1 \leq i < k \leq n$ and $y \in Z$.
	
	Let $W \subset \C^N$ be the smallest affine variety over $\alg$ such that $a \in W$. Let $K \subset \C$ be a number field over which $W$ is defined. If $\dim W = 0$, then we immediately contradict Proposition~\ref{prop:divconjtup}. So we may assume that $\dim W \geq 1$. Note that, by Lemma~\ref{lem:open}, there is some open neighbourhood $U \subset \C^N$ such that $a \in U$ and, for every $a' \in U \cap W$, the divisor condition holds for the modular function $f_{a'}$.
	
	Suppose now that $\dim W = 1$.  Given some $y = (y_1, \ldots, y_n) \in Z$, a condition $p(x, y_i) = 0$ then defines a finite subset of $W$, of size at most $d$ say. Now $y_i$ is a $j$-special point, so is contained in a finite extension of $\Q$ which is ``dihedral'' (see e.g. \cite[p.~191]{GrossZagier85}) and hence, in particular, solvable. So any $a'$ satisfying the equation $p(a', y_i) = 0$ lies in an extension of $K$ which has an index $\leq d$ solvable subextension (and similarly for an equation $q(a', y_i) = 0$ or an equation $p(a', y_i) q(a', y_k) - p(a', y_k) q(a', y_i) = 0$ where $i \neq k$).
	
	Recall that $a$ has at least one transcendental coordinate, say the first coordinate. Let $\pi_1 \colon W \to \C$ be the projection map to the first coordinate. Then $\pi_1$ is an open mapping at $a$, so $\pi_1(W \cap U)$ contains an open neighbourhood $U' \subset \C$ such that $\pi_1(a) \in U'$. If we find $\alpha \in U'$ with Galois group $G$ (over $K$), then any $a' \in W \cap U$ with $\pi_1(a') = \alpha$ will be algebraic and have a Galois group which contains $G$ as a subgroup.
	
	Fix some $b \in U' \cap K(i)$. Let $c$ be any algebraic number whose Galois group over $K(i)$ is isomorphic to the alternating group $A_m$ (for some suitably large $m$); such a $c$ exists by \cite[Corollary~12]{Brink04}. Fix some $r \in \Q$ with $r > 0$ and set $\alpha = b + r c$. Provided $r$ is chosen suitably small, we have that $\alpha \in U'$. The Galois group of $\alpha$ over $K(i)$ is isomorphic to $A_m$, and hence the Galois group of $\alpha$ over $K$ contains a subgroup isomorphic to $A_m$.
	
	Now fix an algebraic $a' \in W \cap U$ with $\pi_1(a') = \alpha$. Then the Galois group of $a'$ over $K$ contains a subgroup isomorphic to $A_m$. In particular, for every $y = (y_1, \ldots, y_n) \in Z$, one has that
	\[p(a', y_i) q(a', y_i) \neq 0\] 
	for all $1 \leq i \leq n$ and
	\[p(a', y_i) q(a', y_k) - p(a', y_k) q(a', y_i) \neq 0 \]
	for all $1 \leq i < k \leq n$. Note also that, for each $y = (y_1, \ldots, y_n) \in Z$, the respective dependence $(*)$ holds with $a'$ in place of $a$, since $a' \in W$. 
	
Define the rational function $R_{a'} \in \alg(t)$ by 
\[R_{a'}(t) = \frac{p(a', t)}{q(a', t)},\]
so that $f_{a'}(z) = R_{a'}(j(z))$. For each $1 \leq k \leq n$, let $A_k$ be the set of $y = (y_1, \ldots, y_k) \in Y^k$ such that the $f_{a'}$-special points $R_{a'}(y_1), \ldots, R_{a'}(y_k)$ are pairwise distinct and the set $\{ R_{a'}(y_1), \ldots, R_{a'}(y_k) \}$ is minimally multiplicatively dependent. Then let
	 \[ \tilde{A_k} = \{ \sigma \in Y : \exists y \in A_k \mbox{ such that } \sigma = \pi_i(y) \mbox{ for some } i=1, \ldots, k \},\]
	 where $\pi_i \colon Y^k \to Y$ is the projection map to the $i$th coordinate as usual.
	 	
Let $y = (y_1, \ldots, y_n) \in Z$. Then, by the above argument, $R_{a'}(y_1), \ldots, R_{a'}(y_n)$ are pairwise distinct $f_{a'}$-special points and 
	\[ R_{a'}(y_1)^{m_1} \cdots R_{a'}(y_n)^{m_n} =1,\]
	for some $m_1, \ldots, m_n \in \Z \setminus \{0\}$. 
	The set $\{R_{a'}(y_1), \ldots, R_{a'}(y_n)\}$ is not necessarily minimally multiplicatively dependent, but we may apply Lemma~\ref{lem:minmult} since $R_{a'}(y_1), \ldots, R_{a'}(y_n)$ are pairwise distinct and $m_1, \ldots, m_n \neq 0$. We thereby obtain that, for each $i =1, \ldots, n$, there exists a minimally multiplicatively dependent subset of $\{R_{a'}(y_1), \ldots, R_{a'}(y_n)\}$ which contains $R_{a'}(y_i)$. In particular, we see that
	 \[y_1, \ldots, y_n \in \bigcup_{k=1}^{n} \tilde{A_k}.\]
	 Consequently, $\tilde{Z} \subset \bigcup_{k=1}^{n} \tilde{A_k}$, and so the set $\bigcup_{k=1}^{n} \tilde{A_k}$ must be infinite.
	 
	 This though cannot happen. Since $a' \in U \cap W$, the divisor condition holds for the modular function $f_{a'}$ by Lemma~\ref{lem:open}. Also, $f_{a'} \in \alg(j)$ since $a'$ is algebraic. Therefore, by Proposition~\ref{prop:divconjtup}, for each $1 \leq k \leq n$, there are only finitely many $k$-tuples $(\sigma_1, \ldots, \sigma_k)$ of pairwise distinct $f_{a'}$-special points such that the set $\{\sigma_1, \ldots, \sigma_k\}$ is minimally multiplicatively dependent. Consequently, the sets $A_1, \ldots, A_n$ must all be finite, since the rational function $R_{a'}$ is finite-to-one. So the sets $\tilde{A_1}, \ldots, \tilde{A_n}$ are all finite too. Thus, the set $\bigcup_{k=1}^{n} \tilde{A_k}$ is also finite, and this gives us the desired contradiction.
	
	So we are left with the case where $\dim W > 1$. For $y \in Z$, define
	\begin{align*} 
		W_y =& \{ w \in W : p(w, y_i) q(w,y_k) -p(w, y_k) q(w, y_i)= 0 \mbox{ for some } 1 \leq i < k \leq n\}\\
		&\cup \{ w \in W : p(w, y_i) q(w,y_i) = 0 \mbox{ for some } i= 1, \ldots, n\}. 
	\end{align*}
	Then $a \notin W_y$ by \eqref{eq:nonzero} and \eqref{eq:distinct}, and so each $W_y$ is a proper subvariety of $W$. The polynomials $p, q$ are of bounded degree, so all the $W_y$ have bounded degree.
	
	Let $W'$ be an irreducible subvariety of $W$ of codimension $1$, defined over $\alg$, and of degree higher than any of the $W_y$. By the theorem of Bertini, such a $W'$ arises as the intersection of $W$ with a sufficiently general hypersurface $H$ in $\C^N$ of large enough degree, defined over $\alg$. The subvariety $W'$ is not contained in any $W_y$. Given $y \in Z$, the corresponding equation $(*)$ holds with any $a' \in W'$ in place of $a$. For a generic $a' \in W'$, the inequalities \eqref{eq:nonzero} and \eqref{eq:distinct} hold for every $y \in Z$. The union of all $H \cap W$ for $H$ as in Bertini's theorem is a constructible, Zariski dense subset of $W$, and therefore contains a Zariski open subset of $W$. We may thus assume that $H$ is such that $H \cap W \cap U$ is non-empty and of codimension $1$ in $W \cap U$.
	
	We now continue this process inductively. Eventually, we thereby obtain a variety $W^*$ which is defined over $\alg$, has dimension $1$, and also satisfies the following two conditions. First, for every $y \in Z$, the corresponding equation $(*)$ holds with any $a' \in W^*$ in place of $a$. Second, for a generic $a' \in W^*$, the inequalities \eqref{eq:nonzero} and \eqref{eq:distinct} hold for every $y \in Z$. One thereby reduces to the above dimension $1$ case, and the proof is complete.
	\end{proof}

\begin{remark}\label{rmk:divnec}
	Observe that the above proof depends crucially on the specific properties of the divisor condition which are embodied in Lemma~\ref{lem:open}. This allows us to specialise to a modular function in $\alg(j)$ that also satisfies the divisor condition. In particular, one could not use the argument in this section to prove a result corresponding to Proposition~\ref{prop:conjtup} for arbitrary non-constant modular functions $f \in \C(j)$ satisfying Condition~\ref{conj:ind} (but not necessarily the divisor condition).
\end{remark}

We end this section by showing that the divisor condition is not necessary for the $n = 1$ case of Theorem~\ref{thm:moddep}, i.e. the case of $f$-special points which are also roots of unity. (Note it is well-known that no $j$-special point is a root of unity, see e.g. \cite[p.~1365]{PilaTsimerman17}.)

\begin{prop}\label{prop:moddep1}
	Let $f \in \C(j)$ be non-constant. Then there are only finitely many $f$-special points $\sigma$ such that $\sigma$ is also a root of unity.
\end{prop}

\begin{proof}
	We treat first the algebraic case. Let $f \in \alg(j)$ be non-constant. Then we may write $f(z) = R(j(z))$ for some rational function $R$ with algebraic coefficients. Suppose $x$ is a $j$-special point such that $R(x)$ is a root of unity. By Kronecker's theorem we have that $h(R(x)) = 0$. Viewing the rational function $R$ as a morphism $\mathbb{P}^1 \to \mathbb{P}^1$ defined over $\alg$, one may use \cite[Theorem~B.2.5]{HindrySilverman00} to obtain that $h(x) \leq c$, where $c$ is a constant depending only on $R$. Write $\Delta$ for the discriminant of the $j$-special point $x$. Then \cite[Lemma~3]{Habegger15} gives that 
	\[h(x) \geq c_1 \log \lvert \Delta \rvert - c_2,\]
	for some absolute constants $c_1, c_2 > 0$. Putting these two inequalities together, we obtain that $\lvert \Delta \rvert$ must be bounded above by a constant depending only on $R$. There are thus only finitely many $j$-special points $x$ such that $R(x)$ is a root of unity. Hence, there are only finitely many $f$-special points which are also roots of unity.
	
Now let $f(z)=R(j(z))$ for an arbitrary non-constant rational function $R$. Suppose that there are infinitely many $f$-special points which are also roots of unity. Since $j$-special points are algebraic, there are then infinitely many $x \in \alg$ such that $R(x) \in \alg$. Hence $R$ must be defined over $\alg$, and so we contradict the above algebraic case of the proposition.
	\end{proof}

\section{The Zilber--Pink conjecture}\label{sec:ZP}

\subsection{The Zilber--Pink setting}\label{subsec:ZPmult}

In this section, we consider how Theorem~\ref{thm:moddep} relates to the Zilber--Pink conjecture. Different versions of this conjecture were formulated by Zilber \cite{Zilber02}, Pink \cite{Pink05}, and Bombieri, Masser and Zannier \cite{BombieriMasserZannier07}.

Let $f \in \C(j)$ be non-constant, so that $f(z) = R(j(z))$ for some non-constant rational function $R$ with coefficients in $\C$. As in Section~\ref{sec:ax}, for $n \geq 1$, let $X =X_{n,n} = Y(1)^n \times \G^n$, $U = \h^n \times \C^n$, and $\pi \colon U \to X$ be given by
\[\pi(z_1, \ldots, z_n, u_1, \ldots, u_n) = (j(z_1), \ldots, j(z_n), \exp(u_1), \ldots, \exp(u_n)).\]
We define (weakly) special subvarieties of $U, X$ as in Definition~\ref{def:special}. Now, as in Subsection~\ref{subsec:conjtup}, let 
\[V= V_n = \{(x_1, \ldots, x_n, t_1, \ldots, t_n) \in X : t_i = R(x_i) \mbox{ for } i=1, \ldots, n\} \subset X.\]
We refer to Definition~\ref{def:atyp} for the definition of an atypical component of $V$. The version of the Zilber--Pink conjecture for $V$ we adopt is the following.

\begin{conj}[Zilber--Pink conjecture]\label{conj:ZP}
	There are only finitely many maximal atypical components of $V$.
\end{conj}

The full Zilber--Pink conjecture is the corresponding statement for an arbitrary subvariety $V$ of a mixed Shimura variety $X$. We restrict ourselves to considering the conjecture for a certain class of subvarieties $V$ of the mixed Shimura variety $Y(1)^n \times \G^n$. We consider the conjecture when $V$ is (the Cartesian product of) the graph of a particular kind of function $\phi \colon Y(1) \to \G$. When $V$ is just the ``diagonal'' subvariety of $X$ 
\[V = \{(x_1, \ldots, x_n, t_1, \ldots, t_n) \in X : x_i = t_i \mbox{ for } i=1,\ldots, n\},\]
then the Zilber--Pink conjecture in this setting was explored in \cite{PilaTsimerman17}. In this section, we extend these results to some cases where $V$ is (the Cartesian product of) the graph of a rational function $R \colon Y(1) \to \G$, rather than just the identity map.

Problems of a related kind were also considered by Pila and Tsimerman \cite{PilaTsimerman19}. They looked at the graph of a map $\phi \colon Y \to E$, where $Y$ is a modular or Shimura curve and $E$ is an elliptic curve. $Y \times E$ is not itself a mixed Shimura variety (unless $E$ has CM), but it is a weakly special subvariety of a mixed Shimura variety. A version of the Zilber--Pink conjecture may thus be formulated for the subvariety of $Y^n \times E^n$ given by the graph of $\phi$.

The Zilber--Pink conjecture does not by itself imply the finiteness of $n$-tuples of distinct $f$-special points that are multiplicatively dependent and minimal for this property. To obtain this result from the Zilber--Pink conjecture, one needs in addition the functional independence statement Condition~\ref{conj:ind}. This is required to prove that every such $n$-tuple leads to a maximal atypical component of $V$. 

\begin{lem}\label{lem:atyp}
Let $f \in \C(j)$ be non-constant. Suppose that Condition~\ref{conj:ind} holds for $f$. An $n$-tuple $\sigma = (\sigma_1, \ldots, \sigma_n)$ of distinct $f$-special points such that the set $\{\sigma_1, \ldots, \sigma_n\}$ is minimally multiplicatively dependent gives rise to an atypical point $\hat{\sigma} \in V$ which is not contained in any atypical component of $V$ of positive dimension.
\end{lem}

\begin{proof}
	The proof is similar to that of \cite[Lemma~6.1]{PilaTsimerman17}. Suppose $\sigma = (\sigma_1, \ldots, \sigma_n)$ is an $n$-tuple of distinct $f$-special points satisfying a non-trivial multiplicative relation and minimal for this property. We may write $\sigma_i = R(x_i)$ for $x_i$ a $j$-special point. 
	
	The point $\hat{\sigma} = (x_1, \ldots, x_n, \sigma_1, \ldots, \sigma_n) \in V$ lies in the intersection of $V$ with a special subvariety $T$ of $X$ of codimension $n+1$. Here $T$ is given by fixing all the $Y(1)$ coordinates and imposing the multiplicative relation satisfied by the $\sigma_i$ on the $\G$ coordinates. Hence, $\hat{\sigma}$ is an atypical point of $V$. We now show that $\hat{\sigma}$ cannot be contained in an atypical component of $V$ of positive dimension.
	
	If $\hat{\sigma}$ were contained in a special subvariety of $X$ defined by two independent multiplicative conditions on the $\sigma_i$ coordinates, then we could eliminate one of these coordinates, contradicting the fact that the set $\{\sigma_1, \ldots, \sigma_n\}$ is minimally multiplicatively dependent.
	
	A special subvariety $M \times \G^n$, where $M$ is a special subvariety of $Y(1)^n$, never intersects $V$ atypically. Similarly, no special subvariety of the form $Y(1)^n \times T$, where $T$ is a special subvariety of $\G^n$, intersects $V$ atypically.
	
	It thus remains to consider special subvarieties of the form $M \times T$, where $M$ is a proper special subvariety of $Y(1)^n$ and $T$ is a special subvariety of $\G^n$ defined by one multiplicative condition. Then $V \cap (M \times T)$ is equal to the set
	\[ \{ (u_1, \ldots, u_n, R(u_1), \ldots, R(u_n)) : (u_1, \ldots, u_n) \in M, (R(u_1), \ldots, R(u_n)) \in T\}.\]
	This would typically have dimension $\dim M -1$. To be atypical, we would thus require that $R(M) \cap \G^n \subset T$. Hence, if $V \cap (M \times T)$ were also positive-dimensional, then Condition~\ref{conj:ind} and the fact that no proper subset of $\{\sigma_1, \ldots, \sigma_n\}$ is multiplicatively dependent implies that $R(M)$ must have two identically equal coordinates. Since no non-constant modular function is invariant under a larger subgroup of $\GL$ than $\Q^{\times} \cdot \SL$, one must then have that $M$ has two identically equal coordinates. But then the  $\sigma_i$ cannot be pairwise distinct, a contradiction.
\end{proof}

As can be seen from the above proof, a multiplicatively dependent $n$-tuple of $f$-special points always gives rise to an atypical component of $V$. However, to show that the resulting atypical component is maximal, one requires also that the points are distinct, the multiplicative dependence is minimal, and $f$ satisfies Condition~\ref{conj:ind}. In such a case, one can then rule out the possibility of there being any positive-dimensional atypical component containing such tuples. It is because we can exclude the positive-dimensional case that our Theorem~\ref{thm:moddep} (and also \cite[Theorem~1.2]{PilaTsimerman17}) is a stronger statement than \cite[Theorem~1.1]{PilaTsimerman19}, where one cannot rule out this case.

With Lemma~\ref{lem:atyp}, it is now easy to prove the following.

\begin{prop}\label{prop:ZPimplies}
	Let $f \in \C(j)$ be non-constant. Suppose that Condition~\ref{conj:ind} and Conjecture~\ref{conj:ZP} hold for $f$. Then, for each $n \geq 1$, there are only finitely many $n$-tuples $(\sigma_1, \ldots, \sigma_n)$ of distinct $f$-special points such that the set $\{\sigma_1, \ldots, \sigma_n\}$ is multiplicatively dependent, but no proper subset of $\{\sigma_1, \ldots, \sigma_n\}$ is multiplicatively dependent.
\end{prop}

\begin{proof}
Let $\sigma = (\sigma_1, \ldots, \sigma_n)$ be an $n$-tuple of distinct $f$-special points satisfying a non-trivial multiplicative relation and such that no proper subset of $\{\sigma_1, \ldots, \sigma_n\}$ is multiplicatively dependent. By Lemma~\ref{lem:atyp} (the proof of which requires only Condition~\ref{conj:ind}), the tuple $\sigma$ gives rise to a maximal atypical component $\{\hat{\sigma}\}$ of $V$. Further, distinct tuples $\sigma$ give rise to distinct points $\hat{\sigma}$. By Conjecture~\ref{conj:ZP}, there are only finitely many maximal atypical components of $V$, and hence there are only finitely many such tuples $\sigma$.
\end{proof}

Thus, for $f \in \C(j)$ non-constant and $n \geq 1$, the finiteness of $n$-tuples of pairwise distinct $f$-special points that are minimally multiplicatively dependent would follow from Condition~\ref{conj:ind} and Conjecture~\ref{conj:ZP} for $f$. Conjecture~\ref{conj:ZP} is not known in general though, and so we must prove the finiteness statement directly, as we did in Section~\ref{sec:tup}.

\subsection{The Zilber--Pink conjecture for $n \leq 2$}\label{subsec:ZP12}

We are not able to prove Conjecture~\ref{conj:ZP} in general. We can though establish some partial results for small $n$. The proofs require the results of Section~\ref{sec:tup}. First, we show that, for $f \in \C(j)$ non-constant, Conjecture~\ref{conj:ZP} holds for $n=1$ thanks to Proposition~\ref{prop:moddep1}. 

\begin{prop}\label{prop:ZP1}
	Let $f \in \C(j)$ be a non-constant modular function. Then Conjecture~\ref{conj:ZP} holds for $n=1$.
\end{prop}

\begin{proof}
	Let $n =1$, so that $X = Y(1) \times \G$ and $V = \{(x,t) \colon t=R(x)\}$, where $R \in \C(t)$ such that $f(z) = R(j(z))$. The special subvarieties of $X$ are either $X$ itself, or have one of the following forms:
	\[Y(1) \times \{\zeta \}, \quad \{x\} \times \G, \quad \{x\} \times \{\zeta \},\]
	where $x$ is a $j$-special point and $\zeta$ is a root of unity. Now $V$ does not intersect $X$ atypically. Nor can $V$ intersect atypically with varieties of form either $Y(1) \times \{\zeta \}$ (since $R$ is non-constant) or $\{x\} \times \G$. So the only atypical components of $V$ arise when, for $x$ a $j$-special point and $\zeta$ a root of unity, the intersection 
	\[V \cap (\{x\} \times \{\zeta \})\]
	is non-empty. This happens just when $\zeta = R(x)$, in which case $\zeta$ is both $f$-special and a root of unity. So it is enough to establish the finiteness of points that are both $f$-special and a root of unity. This though is just Proposition~\ref{prop:moddep1}.
\end{proof}

Now we come to the $n = 2$ case. So, for the remainder of this section, $X = X_{2, 2}$ and $V = V_2 \subset X$. We make the following definitions for convenience.

\begin{definition}\label{def:fin}
	Let $f \in \C(j)$ be non-constant. We say that $f$ satisfies the finiteness condition for pairs if there are only finitely many $2$-tuples $(\sigma_1,  \sigma_2)$ of distinct $f$-special points $\sigma_i$ such that the set $\{\sigma_1, \sigma_2\}$ is multiplicatively dependent and minimal for this property.
\end{definition}

In the remainder of this article, we say that $(x_1, x_2) \in Y(1)^2$ satisfies a modular relation if $\Phi_N(x_1, x_2)=0$ for some $N \geq 1$. Here $\Phi_N$ denotes the $N$th classical modular polynomial. We note (see Remark~\ref{rmk:wkspec}) that $(x_1, x_2)$ satisfies a modular relation if and only if $(x_1, x_2)=(j(z), j(gz))$ for some $z \in \h$, $g \in \GL$. 

\begin{definition}\label{def:modtors}
	Let $f \in \C(j)$ be non-constant. A modular--torsion tuple (for $f$) is a tuple $(x_1, x_2, \zeta_1, \zeta_2) \in V$ such that $(x_1, x_2)$ satisfies a modular relation, $x_1 \neq x_2$, neither $x_1$ nor $x_2$ is $j$-special, and $\zeta_1, \zeta_2$ are both roots of unity.
\end{definition}

We now show that Conjecture~\ref{conj:ZP} holds for $n=2$ if $f \in \C(j)$ is non-constant, satisfies both Condition~\ref{conj:ind} and the finiteness condition for pairs, and is also such that there are only finitely many modular--torsion tuples.

\begin{prop}\label{prop:ZP2}
	Let $f \in \C(j)$ be a non-constant modular function satisfying both Condition~\ref{conj:ind} and the finiteness condition for pairs. Suppose further that there are only finitely many modular--torsion tuples. Then Conjecture~\ref{conj:ZP} holds for $n=2$.
\end{prop}

\begin{proof}
	Here $X = Y(1)^2 \times \G^2$ and $V = \{(x_1,x_2, t_1, t_2) \colon t_1=R(x_1), t_2 =R(x_2)\}$, so $\dim X = 4$ and $\dim V = 2$. We find all maximal atypical components of $V$ by considering the possible special subvarieties $T$ of $X$. We split into cases based on $\codim T$.
	
	\begin{enumerate}[wide, labelwidth=!, labelindent=0pt]
		
		\item The only special subvariety of $X$ of codimension $0$ is $X$ itself, and, once again, $V$ does not intersect $X$ atypically. \\[\parskip]
		
		\item So consider next intersections $V \cap T$, for $T \subset X$ a special subvariety of codimension $1$. So $T$ is defined by just one of: a single fixed $Y(1)$ coordinate; a single modular relation among the $Y(1)^2$ coordinates; a single fixed $\G$ coordinate; or a single multiplicative relation among the $\G^2$ coordinates.
		A component of such an intersection $V \cap T$ is atypical only if it has dimension 
		\[> \dim V + \dim T - \dim X =1\]
		Considering the possible $T$ in turn, we see that this is impossible.\\[\parskip]
		
		\item Now suppose $T \subset X$ is an special subvariety of codimension $2$. A component of the intersection $V \cap T$ is then atypical only if it is positive-dimensional. This clearly rules out all cases where $T$ is defined either by two independent conditions on the $Y(1)^2$ coordinates or by two independent conditions on the $\G^2$ coordinates. So we may assume $T$ is defined by one condition on the $Y(1)^2$ coordinates and one condition on the $\G^2$ coordinates. 
		
		If both conditions are fixed coordinates, then atypical components arise only when $T$ has form either $\{(x, t_1, \zeta, t_2) \colon t_1 \in Y(1), t_2 \in \G\}$ or $\{(t_1, x, t_2, \zeta) \colon t_1 \in Y(1), t_2 \in \G \}$, where $\zeta$ is a root of unity, $x$ is a $j$-special point, and $\zeta = R(x)$. Hence $\zeta$ is both a root of unity and an $f$-special point. There are therefore only finitely many such components by Proposition~\ref{prop:moddep1}. 
		
		If one of the conditions is a fixed coordinate and the other condition is a relation, then the intersection cannot be atypical. The final case to consider is where both the condition on the $Y(1)^2$ coordinates and the condition on the $\G^2$ coordinates are relations. Condition~\ref{conj:ind} then implies that the $\G^2$ coordinates must be equal. Since no non-constant modular function is invariant under a larger subgroup of $\GL$ than $\Q^{\times} \cdot \SL$, one must then have that $T = \{(t_1,t_1,t_2,t_2) \colon t_1 \in Y(1), t_2 \in \G \}$.\\[\parskip]
		
		\item Next let $T$ be a special subvariety of $X$ of codimension $3$. Then $T$ intersects $V$ atypically only if $V \cap T$ is non-empty. The special subvariety $T$ is defined either by two independent modular conditions and one multiplicative condition, or by one modular condition and two independent multiplicative conditions. We may assume that the two conditions of the same type are both fixed coordinates. 
		
		If all three conditions are fixed coordinates, then either both the first $Y(1)^2$ coordinate and the first $\G^2$ coordinate must be fixed, or both the second $Y(1)^2$ coordinate and the second $\G^2$ coordinate must be fixed. Let $x$ be the corresponding fixed $Y(1)^2$ coordinate and $\zeta$ be the corresponding fixed $\G^2$ coordinate. Then $V \cap T$ is non-empty only if $R(x) = \zeta$ is both an $f$-special point and a root of unity. But then atypical components of $V \cap T$ must already be contained in one of the finitely many positive-dimensional atypical components arising from special subvarieties of form either $\{(x, t_1, \zeta, t_2) \colon t_1 \in Y(1), t_2 \in \G\}$ or $\{(t_1, x, t_2, \zeta) \colon t_1 \in Y(1), t_2 \in \G \}$, where $\zeta = R(x)$ is both $f$-special and a root of unity. 
		
		Let $T$ be a special subvariety defined by two fixed $Y(1)^2$ coordinates, say $x_1, x_2$, and a multiplicative relation on the $\G^2$ coordinates. Then $T$ has a non-empty intersection with $V$ only if $R(x_1), R(x_2)$ satisfy this multiplicative relation. The points $R(x_1), R(x_2)$ are $f$-special points, so such $T$ correspond to pairs of $f$-special points that are multiplicatively dependent. The finiteness condition for pairs implies that there are only finitely many such $T$, provided the $f$-special points are distinct and neither is a root of unity. 
		
		If one of the $f$-special points is also a root of unity, then $V \cap T$ is already contained in one of the atypical components described above. If the $f$-special points are not distinct, then either $V \cap T$ is contained in the positive-dimensional atypical component identified above which arises from the special subvariety
		\[\{(t_1,t_1,t_2,t_2) \colon t_1 \in Y(1), t_2 \in \G \}\]
		or $x_1, x_2$ are distinct $j$-special points satisfying $R(x_1)=R(x_2)$. We show that there are only finitely many possibilities for $x_1, x_2$ in the second case.
		
		Suppose there are infinitely many pairs of $j$-special points $(x_1, x_2)$ such that $x_1 \neq x_2$, but $R(x_1)=R(x_2)$. Then the subvariety of $Y(1)^2$ given by 
		\[\{(x,y) \in Y(1)^2 : R(x)=R(y)\}\]
		contains infinitely many points $(x,y)$ with $x, y$ distinct $j$-special points. By Andr\'e's Theorem \cite{Andre98}, there must then be some $g \in \GL$ with the functions $j(z), j(gz)$ distinct and such that $R(j(z))=R(j(gz))$. This though contradicts the fact that no non-constant modular function is invariant under a larger subgroup of $\GL$ than $\Q^{\times} \cdot \SL$. There are thus only finitely many such pairs $(x_1, x_2)$ and hence only finitely many of the resulting atypical components.
		
		If $T$ is defined by one modular relation and two multiplicative conditions (which we may assume are both fixed coordinates), then either $V \cap T$ is contained in one of the finitely many above-described atypical components of positive dimension, or we have a maximal atypical component of the form $\{(x_1, x_2, \zeta_1, \zeta_2)\}$ where $(x_1, x_2)$ satisfies a modular relation, $x_1 \neq x_2$, neither $x_1$ nor $x_2$ is $j$-special, and $\zeta_1, \zeta_2$ are roots of unity. The finiteness of such components is then guaranteed by our assumption on modular--torsion tuples.
		\\[\parskip]
		
		\item Finally, if $T$ is a special subvariety of codimension $4$, then we may assume that all $4$ coordinates are fixed. So $T = \{(x_1, x_2, \zeta_1, \zeta_2)\}$, where $x_1, x_2$ are $j$-special points and $\zeta_1, \zeta_2$ are roots of unity. In this case, $V \cap T$ is non-empty only if $\zeta_1 = R(x_1)$ and $\zeta_2 = R(x_2)$, in which case $\zeta_1, \zeta_2$ are both $f$-special points and roots of unity. In this case $V \cap T$ is contained in one of the atypical components considered previously.
	\end{enumerate}
	
	Consequently, $V$ has only finitely many maximal atypical components, and so Conjecture~\ref{conj:ZP} holds.
\end{proof}

We have thus shown that, for $f \in \C(j)$ non-constant and $n=2$, Conjecture~\ref{conj:ZP} holds if $f$ satisfies both Condition~\ref{conj:ind} and the finiteness condition for pairs and, in addition, there are only finitely many modular--torsion tuples. We are able to prove the finiteness of modular--torsion tuples for non-constant $f \in \alg(j)$ satisfying Condition~\ref{conj:ind}. To do this, we need first the following lemma.

\begin{lem}\label{lem:modpairsdeg}
	Let $f \in \alg(j)$ be a non-constant modular function. For every $\alpha > 1$, there exists a constant $c = c(f, \alpha) > 0$ such that, for each modular--torsion tuple $(x_1, x_2, \zeta_1, \zeta_2)$, one has that $\deg \zeta_1 \leq c (\deg \zeta_2)^\alpha$ and $\deg \zeta_2 \leq c (\deg \zeta_1)^\alpha$.
\end{lem}

\begin{proof}
	Write $f(z) = R(j(z))$ for some rational function $R$ with algebraic coefficients. Let $K$ be a number field over which $R$ is defined. Since $R$ is a rational function, there is some integer $n$ such that $R$ is at most $n$-to-$1$. In this proof, all constants are positive and depend on $f, R, K, n$; any other dependencies will be explicitly indicated.
	
	Define $e \colon \C \to \G$ by $e(z) = \exp(2 \pi i z)$. We let
	\[F_e = \{u \in \C : 0 \leq \re u <1 \}\] 
	be a fundamental domain for the map $e$. The map $e$, restricted to $F_e$, is definable. We let 
	\[ Y = \{ (\nu, g, u) \in F_j \times \GLR \times F_e : g \nu \in F_j, \, R(j(g \nu)) = e(u)\}.\]
	The set $Y$ is definable. We view $Y$ as a definable family of fibres
	\[ Y_{\nu} = \{ (g, u) \in \GLR \times F_e  : g \nu \in F_j, \, R(j(g \nu)) = e(u) \},\]
	where $\nu \in F_j$.
	
	In particular, since the Counting Theorem of \cite{PilaWilkie06} is uniform in definable families, for every $\epsilon > 0$, there exists a constant $c(\epsilon) > 0$ such that, for every $\nu \in F_j$, either $N(Y_{\nu}, T) \leq c(\epsilon) T^{\epsilon}$ for all $T \geq 1$, or else $Y_{\nu}$ contains a positive-dimensional, connected semialgebraic set. Here $N(Y_{\nu}, T)$ denotes the number of rational points in the set $Y_{\nu}$ which have height $\leq T$.
	
	Suppose then that the lemma is false. Fix some $\alpha > 1$ for which it fails. Then, for every $M \geq 1$, there is some modular--torsion tuple $(x_1, x_2, \zeta_1, \zeta_2)$ with either $\deg \zeta_2 > M (\deg \zeta_1)^\alpha$ or $\deg \zeta_1 > M (\deg \zeta_2)^\alpha$. Fix some suitably large $M$ and such a modular--torsion tuple $(x_1, x_2, \zeta_1, \zeta_2)$. Write $d_i = \deg \zeta_i$. Without loss of generality, we may assume that $d_1 \leq d_2$, and so we must have that $d_2 > M d_1^{\alpha}$. Write $m_i$ for the order of $\zeta_i$. So $d_i = \phi(m_i)$. Let $\nu \in F_j$ be such that $j(\nu) = x_1$. We show that $(x_1, x_2, \zeta_1, \zeta_2)$ gives rise to a rational point of the set $Y_{\nu}$ in the following way.
	
	We have that $\zeta_2 = e(q)$, where $q = a/ m_2$ for some $a \in \Z$ with $0 < a < m_2$ and $\gcd(a, m_2)=1$. In particular, $H(q) = m_2$. Let $E_i$ be an elliptic curve with $j$-invariant $x_i$. The modular relation satisfied by $x_1, x_2$ implies that the elliptic curves $E_1$, $E_2$ are isogenous. We will bound the degree of this isogeny.
	
	First, we bound the degrees of $x_1, x_2$. Write
	\[R(t) = \frac{p(t)}{q(t)},\]
	where $p(t), q(t) \in K[t]$. Let $d = [K \colon \Q]$ and $l = \max\{\deg p, \deg q\}$. Note that $x_i$ is a root of the non-zero polynomial $f_i(t) = p(t)^{m_i} - q(t)^{m_i}$. The polynomial $f_i$ has degree $ \leq l m_i$ and coefficients in $K$. Hence
	\[[\Q(x_i) \colon \Q] \leq [K(x_i) \colon \Q] = [K(x_i) \colon K] [K \colon \Q] \leq l m_i d.\]
	So $x_i$ has degree bounded by $c_1 m_i$. 
	
	Next, we bound the logarithmic heights of $x_1, x_2$. Viewing the rational function $R$ as a morphism $\mathbb{P}^1 \to \mathbb{P}^1$, one may use \cite[Theorem~B.2.5]{HindrySilverman00} and the fact that $R(x_i) = \zeta_i$ is a root of unity to obtain that $h(x_1), h(x_2) \leq c_2$.
	
	We then use \cite[(5.8)]{PilaTsimerman17} to bound the semistable Faltings height $\falt(E_1)$ of the elliptic curve $E_1$. We obtain that
	\[\falt(E_1) \leq c_3 \max \{1, h(x_1)\}.\]
	In particular, the above bound on $h(x_1)$ thus implies that $\falt(E_1) \leq c_4$. Combining this bound on the Faltings height with the above degree bounds on $x_1, x_2$, we may use \cite[(5.10)]{PilaTsimerman17} to deduce that there is an isogeny between $E_1$ and $E_2$ of degree $N \leq c_5 \max \{m_1, m_2\}^5$. Consequently, by \cite[Lemma~5.2]{HabeggerPila12}, there exists $g \in \GL$ such that $g \nu \in F_j$, $j(g \nu) = x_2$, and the height of $g$ (regarded as the vector of its entries) is bounded by $c_6 \max\{m_1, m_2\}^{50}$. 
	
	The modular--torsion tuple $(x_1, x_2, \zeta_1, \zeta_2)$ thus gives rise to the rational point $(g, q) \in Y_{\nu}$, and this point has height $\leq c_7 \max\{m_1, m_2\}^{50}$. Since $\phi(m_i) = d_i$, we may use the elementary lower bound $\phi(m_i) \geq \sqrt{m_i/2}$ to obtain that the height of this rational point is $\leq c_8 d_2^{100}$.
	
	Now consider the conjugates of $(x_1, x_2, \zeta_1, \zeta_2)$ over the field $K(\zeta_1)$. Observe that
	\begin{align*} 
		[K(\zeta_1, \zeta_2) : K(\zeta_1)][K(\zeta_1) : \Q] = [K(\zeta_1, \zeta_2) : \Q] \geq [\Q(\zeta_2) : \Q] = d_2,
	\end{align*}
	so
	\begin{align*}
		[K(\zeta_1, \zeta_2) : K(\zeta_1)] &\geq \frac{d_2}{[K(\zeta_1) : \Q]}\\
		&= \frac{d_2}{[K(\zeta_1) : \Q(\zeta_1)][\Q(\zeta_1) : \Q]}\\
		&\geq \frac{d_2}{[K : \Q] d_1}.
	\end{align*}
	Recall that $d_2 > M d_1^\alpha$. Hence, there are at least
	\[[K(\zeta_1, \zeta_2) : K(\zeta_1)] \geq c_9 M^{1/\alpha} d_2^{(\alpha - 1) / \alpha} \]
	conjugates of $(x_1, x_2, \zeta_1, \zeta_2)$ over the field $K(\zeta_1)$.
	
	We may enumerate the conjugates as $(x_1^{(i)}, x_2^{(i)}, \zeta_1, \zeta_2^{(i)})$. Note that $R(x_1^{(i)}) = \zeta_1$ for all of these conjugates. Hence, there are at most $n$ distinct coordinates $x_1^{(i)}$ among these conjugates since $R$ is at most $n$-to-$1$. Let $\nu_1, \ldots, \nu_k \in F_j$ be such that $j(\nu_1), \ldots, j(\nu_k)$ are all distinct and give all the possible $x_1^{(i)}$ coordinates. Note that $k \leq n$.
	
	Now each distinct conjugate $(x_1^{(i)}, x_2^{(i)}, \zeta_1, \zeta_2^{(i)})$ gives rise, in the same way as above, to a distinct rational point which lies on one of the definable sets $Y_{\nu_1}, \ldots, Y_{\nu_k}$. Further, these rational points all have height bounded by $c_8 d_2^{100}$ since every $\zeta_2^{(i)}$ is again a root of unity of order $m_2$. Thus, there must be at least one $r \in \{1,\ldots, k\}$ such that $Y_{\nu_r}$ has at least $c_{10} M^{1/\alpha} d_2^{(\alpha - 1) / \alpha}$ rational points of height $\leq c_8 d_2^{100}$.
	
	We now apply the above-stated Counting Theorem with $\epsilon = (\alpha -1)/(100 \alpha)$ and $T = c_8 d_2^{100}$. Provided $M$ is suitably large (which we may always assume), there must then exist some $r \in \{1, \ldots, k\}$ such that $Y_{\nu_r}$ contains a positive-dimensional, connected, semialgebraic set. Fix such an $r$ and let $S$ be the corresponding semialgebraic set. 
	
	Define the map $\Theta \colon \GLR \times \C \to \h \times \C$ by $(g, u) \mapsto (g \nu_r, u)$. In particular, the map $\Theta$ is semialgebraic. The set $\Theta(S) \subset \h \times \C$ is positive-dimensional and semialgebraic. Define the map $\pi_e \colon \h \times \C \to Y(1) \times \G = X_1$ by $\pi_e = (j, e)$. Observe that $\Theta(S) \subset \pi_e^{-1}(V_1)$ since $S \subset Y_{\nu_r}$. Arguing as in the proof of Proposition~\ref{prop:conjtup}, we may find a complex algebraic subvariety $W \subset \C^2$ and a positive-dimensional complex-analytically irreducible component $A \subset (\h \times \C) \cap W$ such that $A \subset \pi_e^{-1}(V_1)$. Ax--Schanuel for the map $\pi_e$ (which follows from the Ax--Schanuel results in Section~\ref{sec:ax}) then implies that $V_1$ must contain a positive-dimensional weakly special subvariety. The only weakly special subvarieties of $V_1$ are points though, because $V_1 = \{ (x, t) \in X_1 : R(x) = t\}$ and the map $R$ is non-constant. We thus obtain a contradiction, and the result is proved.
\end{proof}

\begin{prop}\label{prop:modpairs}
	Suppose that $f \in \alg(j)$ is a non-constant modular function satisfying Condition~\ref{conj:ind}. Then there are only finitely many modular--torsion tuples.
\end{prop}

\begin{proof}
	Write $f(z) = R(j(z))$ for some rational function $R$ with algebraic coefficients. Let $K$ be a number field over which $R$ is defined. In this proof, all constants are positive and may depend possibly on $f, R, K$; any other dependencies will be explicitly indicated.
	
	Observe that the proof of Proposition~\ref{prop:ZP2} implies that every modular--torsion tuple $(x_1, x_2, \zeta_1, \zeta_2)$ gives rise to a maximal atypical component $\{(x_1, x_2, \zeta_1, \zeta_2)\}$ of $V$. In particular, no modular--torsion tuple is contained in a positive-dimensional atypical component of $V$.
	
	We define the complexity $\Delta$ of a modular--torsion tuple $(x_1, x_2, \zeta_1, \zeta_2)$ by $\Delta = \min \{ \deg \zeta_1, \deg \zeta_2\}$. The degree bound in Lemma~\ref{lem:modpairsdeg} and the fact that $R$ is finite-to-one together imply that, for any $C > 0$, there are only finitely many modular--torsion tuples with complexity $\leq C$. Suppose then, for contradiction, that there are infinitely many modular--torsion tuples. In particular, there are modular--torsion tuples of arbitrarily large complexity.
	
	Now let
	\begin{align*} Y = \{&(g, z_1, z_2, u_1, u_2) \in \GLR \times F_j^2 \times F_e^2 : \\
		&g z_1 = z_2, \, R(j(z_1)) = e(u_1), \, R(j(z_2)) = e(u_2)\}, \end{align*}
	and let
	\[ Z = \{(g, u_1, u_2) : \exists z_1, z_2 \, (g, z_1, z_2, u_1, u_2) \in Y\}\]
	be its projection. Both sets are definable.
	
	Suppose $(x_1, x_2, \zeta_1, \zeta_2)$ is a modular--torsion tuple of complexity $\Delta$. We show that this tuple leads to a rational point on $Z$ of bounded height. Say $\zeta_1$ is a primitive $m$th root of unity and $\zeta_2$ is a primitive $n$th root of unity. So $\deg \zeta_1 = \phi(m_1)$ and $\deg \zeta_2 = \phi(m_2)$. Write $d_i = \deg \zeta_i$. Without loss of generality, we assume that $d_2 \geq d_1$.
	
	We have that $\zeta_1 = e(k/m_1)$ for some $0 \leq k < m_1$ with $\gcd(k, m_1) = 1$ and $\zeta_2 = e(l/m_2)$ for some $0 \leq l < m_2$ with $\gcd(l, m_2) = 1$. In particular, $H(k/m_1) = m_1$ and $H(l/m_2) = m_2$.
	
	Let $E_1, E_2$ be elliptic curves with $j$-invariants $x_1, x_2$ respectively. As in the proof of Lemma~\ref{lem:modpairsdeg}, we may then bound the degrees and logarithmic heights of $x_1, x_2$. Once again, we obtain that the degree of an isogeny between $E_1, E_2$ is bounded by $c_1 (\max \{m_1, m_2\})^{c_2}$. Let $\tau_1, \tau_2 \in F_j$ such that $x_i = j(\tau_i)$. By \cite[Lemma~5.2]{HabeggerPila12}, there then exists $g \in \GL$ such that $g \tau_1 = \tau_2$ and the height of $g$ (viewed as a vector of its entries) is bounded by $c_3 (\max\{m_1, m_2\})^{c_4}$.
	
	The modular--torsion tuple $(x_1, x_2, \zeta_1, \zeta_2)$ thus gives rise to the rational point $\sigma = (g, k/m_1, l/m_2) \in Z$.  The height bound on $g$, Lemma~\ref{lem:modpairsdeg}, and the elementary estimate that $\phi(m_i) \geq \sqrt{m_i/2}$ together imply that this point $\sigma$ has height bounded by $c_5 \Delta^{c_6}$.
	
	Observe that
	\[[K(\zeta_i) : K][K: \Q] = [K(\zeta_i) : \Q] \geq d_i \geq \Delta\]
	since $\Delta = \min \{d_1, d_2\}$, and so
	\[ [K(\zeta_i) : K] \geq c_7 \Delta. \]
	The modular--torsion tuple $(x_1, x_2, \zeta_1, \zeta_2)$ thus has $ \geq c_7 \Delta$ conjugates over $K$. Each of these conjugates $(x_1^{(r)}, x_2^{(r)}, \zeta_1^{(r)}, \zeta_2^{(r)})$ is again a modular--torsion tuple. In particular, note that each $\zeta_i^{(r)}$ is again a root of unity of order $m_i$. Each distinct conjugate $(x_1^{(r)}, x_2^{(r)}, \zeta_1^{(r)}, \zeta_2^{(r)})$ thus gives rise, in the same way as above, to a rational point $(g_r, k_r/m_1, l_r/m_2) \in Z$ which has height bounded by $c_5 \Delta^{c_6}$. Moreover, there must also be $\geq c_7 \Delta$ distinct coordinates $k_r/m_1$ and $\geq c_7 \Delta$ distinct coordinates $l_i/m_2$ among the resulting rational points.
	
	We now apply the Counting Theorem in the form of \cite[Theorem~3.5]{Pila09} to the set $Z$, with $\epsilon = 1/{2 c_6}$ say. We may write the resulting basic block families as $W_i$, where $i=1, \ldots, J$, for some constant $J$. In particular, the rational points of $Z$ with height at most $T=c_5 \Delta^{c_6}$ are contained in $\leq c_8 T^{1/2 c_6} = c_9 \Delta^{1/2}$ basic blocks, each of which is a fibre of one of $W_1, \ldots, W_J$. 
	
Recall that the structure $\R_{\mathrm{an}, \exp}$ has analytic cell decomposition \cite{vandenDriesMiller94}. We may thus decompose each of the basic block families $W_1, \ldots, W_J$ into finitely many analytic cells $P_i$, and moreover this may be done in such a way that it induces an analytic cell decomposition of each fibre of the basic block family over its base.
	
	As shown above, a modular--torsion tuple of complexity $\Delta$ gives rise, via its $K$-conjugates, to a collection of rational points $(g, u_1, u_2)$ on $Z$ which are all of height $\leq c_5 \Delta^{c_6}$ and among which there $\geq c_7 \Delta$ distinct $u_1$ coordinates and $\geq c_7 \Delta$ distinct $u_2$ coordinates. Provided that $\Delta$ is suitably large, we thus see that at least one of the analytic cells $P_i$ must contain one of these rational points and also have non-constant projection to both its $F_e$ coordinates. Fix such a cell $P$ and such a rational point $p =(g_p, u_{1, p}, u_{2, p}) \in P$.
	
	The cell $P$ is contained in a basic block $B$ that is itself contained in $Z$. As in the proof of \cite[Proposition~3.4(1)]{Pila09}, one may then find an open neighbourhood $\Omega_\R \subset \GLR \times \C^2$ of $p$ such that the intersection $\Omega_\R \cap B$ is a real semialgebraic set, which we denote $S$. Note that $p \in S$.
	
	Provided that $\Delta$ is sufficiently large, we must also have that
	\[ e(u_{1, p}), e(u_{2, p}) \notin \{ f(w) : w \in \h \mbox{ and } f'(w)=0\}.\]
	Therefore, the function $f$ is locally invertible at $e(u_{1,p}), e(u_{2,p})$. Given a sufficiently small open neighbourhood $\Omega \subset \mathrm{GL}_2(\C) \times \C^2$ of $p$, we may therefore define an analytic subvariety $A \subset \Omega$ by
	\begin{align*} 
		A = \{&(g, e(u_1), e(u_2)) \in \Omega : g= \begin{pmatrix}
			a & b\\
			c & d
		\end{pmatrix} \mbox{ and }\\
 &\prod_{h_1, h_2} ((c h_1(e(u_1)) + d)h_2(e(u_2)) - (a h_1(e(u_1))+b))=0\},
\end{align*}
	where the product runs over all the local inverses $h_1, h_2$ of $f$ at $e(u_{1, p}), e(u_{2,p})$ respectively which hit the fundamental domain $F_e$. Without loss of generality, we may assume that $\Omega_\R = \Omega \cap (\GLR \times \C^2)$, and thus $S \subset A$.
	
	As in the proof of Proposition~\ref{prop:conjtup}, we may now apply the results of \cite{AdamusRand13}. There exists an open neighbourhood $\Omega' \subset \Omega$ of $p$ and a set $Q \subset \Omega'$, where $Q$ may be written as a finite union of irreducible Nash subsets of $\Omega'$ which each contain $p$, such that $p \in \Omega' \cap S \subset Q \subset A$. 
	
	If every complex-analytically irreducible component of $Q$ projects constantly to one or more of its $F_e$ coordinates, then the cell $P$ would have constant projection to (at least) one of its $F_e$ coordinates by real analytic continuation. This cannot happen. Hence, there must exist a complex algebraic subvariety $W \subset \mathrm{GL}_2(\C) \times \C^2$ such that there is a complex-analytically irreducible component $D \subset \Omega' \cap W$ such that $p \in D \subset A$ and $D$ projects non-constantly to both its $F_e$ coordinates.
	
	Denote by $\nu \colon W^{\nu} \to W$ the normalisation of the variety $W$. There exists an open neighbourhood in $\nu^{-1}(W \cap A)$ of some preimage $p^{\nu}$ of $p$. We may choose a point $q$ in this neighbourhood such that, writing $\nu(q) = (g, u_1, u_2)$, we have that $u_1 \neq u_{1, p}$ and $u_2 \neq u_{2, p}$. By \cite[Corollary~1.9]{CharlesPoonen16}, we may then find an irreducible complex algebraic curve $T \subset W^{\nu}$ which passes through the points $p^{\nu}, q$.
	
	The Zariski-closure of the set
	\[ \{ (z, gz, u_1, u_2) : z \in \C, (g, u_1, u_2) \in \nu(T)\}\]
	is then a complex algebraic surface, which we denote $H$. Let $\pi_e = (j, j, e, e) \colon \h^2 \times \C^2 \to Y(1)^2 \times \G^2$, where $e(z) = \exp(2 \pi i z)$ as in the proof of Lemma~\ref{lem:modpairsdeg}. There exists a positive-dimensional complex-analytically irreducible component $H_0$ of $H \cap \pi_e^{-1}(V)$ which has non-constant projection to both its $u_i$ coordinates and also contains a point $(z, g_p z, u_{1, p}, u_{2,p})$ whose image under $\pi_e$ is a modular--torsion tuple.
	
	The Ax--Schanuel results of Section~\ref{sec:ax} (applied with $\pi_e$ in place of $\pi$) then imply that there exists a weakly special subvariety $W_0$ of $U$ such that $H_0 \subset W_0 \subset \pi_e^{-1}(V)$. Since this weakly special subvariety $W_0$ has no constant coordinates, it is in fact a special subvariety. It must also have codimension at least two, in order to be contained in $\pi_e^{-1}(V)$. Taking the image under $\pi_e$, we thus obtain a positive-dimensional atypical component of $V$ which contains $\pi_e(z, g_p z, u_{1, p}, u_{2,p})$. This is a positive-dimensional atypical component of $V$ which contains a modular--torsion tuple. However, we know already that there are no such components. We thus obtain the desired contradiction, and so the proof is complete.
\end{proof}

Recall that Proposition~\ref{prop:conjtup} implies that if $f \in \alg(j)$ is non-constant and satisfies Condition~\ref{conj:ind}, then the finiteness condition for pairs holds for $f$. Therefore, one may deduce from Propositions~\ref{prop:ZP2} and \ref{prop:modpairs} the following case of Conjecture~\ref{conj:ZP}.

\begin{prop}\label{prop:uncondZP2}
	Let $f \in \alg(j)$ be non-constant. Suppose that $f$ satisfies Condition~\ref{conj:ind}. Then Conjecture~\ref{conj:ZP} holds for $n = 2$.
\end{prop}

\section{An extension to finite rank}\label{sec:finrk}

In this section we prove Theorem~\ref{thm:gamdep}. As in Section~\ref{sec:tup}, we will prove a stronger conditional result under the assumption of Condition~\ref{conj:ind}. Theorem~\ref{thm:gamdep} then follows from Proposition~\ref{prop:conjgam} as it corresponds to those cases where the divisor condition holds and Condition~\ref{conj:ind} is thus known via Theorem~\ref{thm:modind}.

\begin{prop}\label{prop:conjgam}
		Let $f \colon \h \to \C$ be a non-constant modular function such that $f \in \alg(j)$. Assume Condition~\ref{conj:ind} holds for $f$. Let $n \geq 1$ and $\Gamma \leq \G$ be of finite rank. Then there exist only finitely many $n$-tuples $(\sigma_1, \ldots, \sigma_n)$ of distinct $f$-special points such that the set $\{\sigma_1, \ldots, \sigma_n \}$ is $\Gamma$-dependent, but no proper subset of $\{\sigma_1, \ldots, \sigma_n \}$ is $\Gamma$-dependent.
\end{prop}

Fix such a function $f$ and such a group $\Gamma$. If $\Gamma$ has rank $0$, then the result follows immediately from Proposition~\ref{prop:conjtup}. So we may assume that $\Gamma$ has positive rank. As usual, we write $f(z) = R(j(z))$ for $R$ some rational function with algebraic coefficients. Let $K$ be a number field containing the coefficients of $R$.

 The group $\Gamma$ is of finite rank, so there exists $\Gamma_0 \leq \Gamma$ finitely generated such that for every $\gamma \in \Gamma$ there exists $m \geq 1$ such that $\gamma^m \in \Gamma_0$. Thus every $\Gamma$-dependent tuple is also $\Gamma_0$-dependent. Therefore, we may and do assume that $\Gamma$ is finitely generated.

Further, $f$-special points are algebraic. So if $\prod \sigma_i^{a_i} = \gamma \in \Gamma$ for some $f$-special points $\sigma_i$ and $a_i \in \Z$, then $\gamma \in \Gamma \cap \alg$. Now $\Gamma \cap \alg$ is a subgroup of the finitely generated abelian group $\Gamma$, so is itself finitely generated. Replacing $\Gamma$ with $\Gamma \cap \alg$ as necessary we may and do assume as well that $\Gamma$ is generated by algebraic elements.

We thus consider $\Gamma = \langle b_1, \ldots, b_k \rangle$, where $k \geq 1$ and every $b_i \in \alg$. A tuple of $f$-special points $(\sigma_1, \ldots, \sigma_n)$ is then $\Gamma$-dependent if some relation 
\[\prod_{i=1}^n \sigma_i^{a_i} = \prod_{i=1}^k b_i^{\alpha_i}\]
holds with $a_i, \alpha_i \in \Z$ and the $a_i$ not all zero. We may also assume that the set $\{b_1, \ldots, b_k\}$ is multiplicatively independent; in particular, no $b_i$ is a root of unity. We let $K_0 = K(b_1, \ldots, b_k)$.

Now fix $n \geq 1$. Let $X =X_{n,n+k}=Y(1)^n \times \G^{n+k}$, and
\[V = \{(x_1, \ldots, x_n, t_1, \ldots, t_n, b_1, \ldots, b_k) \in X \colon  t_i = R(x_i) \mbox{ for } i=1,\ldots, n\}.\] 
We refer to Section~\ref{sec:ax} for the definition of a (weakly) special subvariety of $X$. 

Now for the proof of Proposition~\ref{prop:conjgam}. Constants $c, c'$ will be positive and depend only on $f$, $n$, and $\Gamma$, but will vary between occurrences and are in general non-effective. We note that the proofs of Lemmas~\ref{lem:degspec} and \ref{lem:conj} both generalise straightforwardly when we replace $K$ with $K_0$ in their statements.

\begin{proof}[Proof of Proposition~\ref{prop:conjgam}]
	For a tuple $\sigma = (\sigma_1, \ldots, \sigma_n)$ of $f$-special points, define the complexity $\Delta(\sigma)$ of $\sigma$ to be $\Delta(\sigma) = \max \{\lvert \Delta(\sigma_1) \rvert, \ldots, \lvert \Delta(\sigma_n)\rvert \}$. Call an $n$-tuple $(\sigma_1, \ldots, \sigma_n)$ of distinct $f$-special points such that the set $\{\sigma_1, \ldots, \sigma_n \}$ is $\Gamma$-dependent, but no proper subset of $\{\sigma_1, \ldots, \sigma_n \}$ is $\Gamma$-dependent, a $\Gamma$-tuple. We show that there are only finitely many $\Gamma$-tuples.
	
	Suppose not. Then there are $\Gamma$-tuples $\sigma$ of arbitrarily large complexity $\Delta(\sigma)$. Fix $\nu = (\nu_1, \ldots, \nu_k) \in F_{\exp}^k$ a preimage of $(b_1, \ldots, b_k)$. We let 
	\[Y = \{ (z, u \, \nu, r, s) \in F_j^{n} \times F_{\exp}^{n+k} \times \R^{n+k} \times \R \colon R(j(z)) = \exp(u), r \cdot (u \, \nu) = 2\pi i s\},\]
	where $F_j, F_{\exp}$ are the standard fundamental domains for the actions of $\SL$ on $\h$ and $2 \pi i \Z$ on $\C$ respectively (see Section~\ref{sec:tup}). Then let $Z$ be the projection of $Y$ to $F_j^n \times \R^{n+k} \times \R$. Note that $Y, Z$ are definable.
	
	Fix a $\Gamma$-tuple $\sigma = (\sigma_1, \ldots, \sigma_n)$ of complexity $\Delta = \Delta(\sigma)$. We may write $\sigma_i = R(x_i)$ for some $j$-special point $x_i$. The tuple $\sigma$ satisfies some multiplicative relation
	\[\prod_{i=1}^n \sigma_i^{a_i} \prod_{i=1}^k b_i^{\alpha_i} =1\]
	for $a_i, \alpha_i \in \Z$, with the $a_i$ not all zero. The point 
	\[\hat{\sigma} = (x_1, \ldots, x_n, \sigma_1, \ldots, \sigma_n, b_1, \ldots, b_k) \in Y(1)^n \times \G^{n+k}\]
	 has a preimage $(z, u \, \nu) \in F_j^{n} \times F_{\exp}^{n+k}$. This gives rise to a point $(z , u \, \nu, \beta \, \beta', \delta) \in Y$. Here the $\beta, \beta', \delta$ coordinates are rational integers, with the $\beta_i$ not all zero, which record the multiplicative dependence of the $\sigma_i, b_i$; that is,
	 \[\sum _{i=1}^n \beta_i u_i + \sum_{i=1}^{k} \beta_i' \nu_i = 2 \pi i \delta.\]
	
	We may apply Lemma~\ref{lem:htmultdep} to some minimally multiplicatively dependent set $S$, where $\{\sigma_1, \ldots, \sigma_n\} \subset S \subset \{\sigma_1, \ldots, \sigma_n, b_1, \ldots, b_k\}$. The integers $\beta_i, \beta_i'$ may thus be chosen such that 
	\[\lvert \beta_i \rvert \leq c d^{n+k} (\log d) \prod_{\substack{j=1\\ j \neq i}}^{n} h(\sigma_j)\]
	and 
	\[\lvert \beta_i' \rvert \leq c d^{n+k} (\log d) \prod_{j=1}^{n} h(\sigma_j).\]
	Here $d \geq 2$ is the degree of a number field containing $\sigma_1, \ldots, \sigma_n, b_1, \ldots, b_k$. (We absorb the dependency on the logarithmic heights of the $b_i$ into our constant.) Using Lemmas~\ref{lem:htspec} and \ref{lem:degspec}, we may thus bound $\lvert \beta_i \rvert$ and $\lvert \beta_i' \rvert$ by $c \Delta^{n(n+k)}$. Observe that for $u_i, \nu_i \in F_{\exp}$, the imaginary part of $u_i, \nu_i$ is bounded by $2 \pi$ in absolute value. Then, using the relation
	\[ \sum _{i=1}^n \beta_i u_i + \sum_{i=1}^{k} \beta_i' \nu_i = 2 \pi i \delta,\]
we may also bound $\lvert \delta \rvert$ by $c \Delta^{n(n+k)}$. Since $\beta_i, \beta_i', \delta$ are rational integers, their heights are also bounded by $c \Delta^{n(n+k)}$ since $H(l) = \lvert l \rvert$ for $l \in \Z \setminus \{0\}$.
	
	The point $(z, u \, \nu, \beta \, \beta', \delta)$ projects to a point $(z, \beta \, \beta', \delta) \in Z$, which is quadratic in the $F_j$ coordinates and integral in the $\R$ coordinates.  The height of the $z$ coordinates may be bounded by $c \Delta$ thanks to Lemma~\ref{lem:htpre}. Combining this with the height bounds from the previous paragraph, we see that the height of the point $(z, \beta \, \beta', \delta)$ is thus bounded by $c \Delta^{n(n+k)}$. The $\Gamma$-tuple $\sigma$ thus gives rise to a quadratic point of $Z$ with height bounded by $c \Delta^{n(n+k)}$.
	
	By the aforementioned generalisation of Lemma~\ref{lem:degspec} with $\epsilon = 1/4$, a $\Gamma$-tuple $\sigma$ of complexity $\Delta$ has $\geq c \Delta^{1/4}$ distinct conjugates over $K_0$, provided $\Delta$ is large enough. By the similar generalisation of Lemma~\ref{lem:conj} and the fact that $b_1, \ldots, b_k \in K_0$, any $K_0$-conjugate of such a $\Gamma$-tuple is again a $\Gamma$-tuple and is also of complexity $\Delta$. In particular, the $K_0$-conjugates all satisfy the same $\Gamma$-dependence. 
	
	Therefore, each $K_0$-conjugate $\tilde{\sigma}$ of $\sigma$ also gives rise, in the same way as $\sigma$, to a quadratic point $(\tilde{z}, \beta \, \beta', \tilde{\delta})$ of $Z$ with height bounded by $c \Delta^{n(n+k)}$. Note that that the $\beta, \beta'$ coordinates are the same for each conjugate. Therefore, a $\Gamma$-tuple $\sigma$ of sufficiently large complexity $\Delta$ gives rise to at least $c' T^{1/4n(n+k)}$ quadratic points on $Z$ with height at most $T = c \Delta^{n(n+k)}$, each corresponding to a distinct $\Gamma$-tuple. 
	
	Let $Z'$ be the image of $Z$ under the map $(z, r, s) \mapsto (z, 2 \pi i s)$. Denote by $Y_{(\beta_1, \ldots, \beta_n, \beta_1' \ldots, \beta_k')}$ the fibre of $Y$ over the point $(\beta_1, \ldots, \beta_n, \beta_1' \ldots, \beta_k')$. We may now apply the Counting Theorem in the form of \cite[Corollary~7.2]{HabeggerPila16} to $Y_{(\beta_1, \ldots, \beta_n, \beta_1' \ldots, \beta_k')}$, and then complete the argument as in the proof of Proposition~\ref{prop:conjtup}. We will only sketch the remainder of the proof, and refer to the proof of Proposition~\ref{prop:conjtup} for details.
	
	 Arguing as in the proof of of Proposition~\ref{prop:conjtup}, if the complexity $\Delta$ of $\sigma$ is suitably large, then we obtain a complex algebraic subvariety $W \subset \C^{n+1}$ and a complex-analytically irreducible component $A \subset (\h^n \times \C) \cap W$ such that $A$ has non-constant projection to its $\h^n$ coordinates, $A$ contains a point of $Z'$ which corresponds to a $K_0$-conjugate of $\sigma$, and $A \subset \tilde{V}$ where
	\[\tilde{V} = \{(z, t) \in \h^n \times \C : \prod_{i=1}^n f(z_i)^{\beta_i} \prod_{i=1}^k b_i^{\beta_i'} = \exp(t)\}.\]
	By Ax--Schanuel, we thus obtain that there are weakly special subvarieties $W_1 \subset \h^n$ and $W_2 \subset \C$ such that $A$ is contained in $W_1 \times W_2$ and $W_1 \times W_2 \subset \tilde{V}$. As in the proof of Proposition~\ref{prop:conjtup}, one must have that $W_2 = \{2 \pi i \delta'\}$ for some $\delta' \in \Z$.
	
	Therefore, $W_1$ is a positive-dimensional weakly special subvariety of $\h^n$ which is contained in the set 
	\[\{z \in \h^n : \prod_{i=1}^n f(z_i)^{\beta_i} \prod_{i=1}^k b_i^{\beta_i'} = 1\}.\]
	In addition, $W_1$ contains a pre-image $(\tau_1, \ldots, \tau_n)$ of some $\Gamma$-tuple $(f(\tau_1), \ldots, f(\tau_n))$. In particular, $f(W_1)$ has no two identically equal coordinates, since the $f(\tau_i)$ are pairwise distinct. In addition, all the $\beta_i$ are non-zero, since the set $\{f(\tau_1), \ldots, f(\tau_n)\}$ is minimally $\Gamma$-dependent. Taking the image of $W_1$ under $f$, we therefore obtain a multiplicative dependence modulo constants among some pairwise distinct $\GL$-translates of $f$. This contradicts Condition~\ref{conj:ind}, and so we are done.
	\end{proof}

For the $n=1$ case, Proposition~\ref{prop:conjgam} holds unconditionally. 

\begin{prop}\label{prop:gam1}
	Let $f \colon \h \to \C$ be a non-constant modular function such that $f \in \alg(j)$. Let $\Gamma \leq \G$ be of finite rank. Then there exist only finitely many $f$-special points $\sigma$ such that the set $\{\sigma\}$ is $\Gamma$-dependent.
\end{prop}

\begin{proof}
	If $\Gamma$ has rank $0$, then this follows from Proposition~\ref{prop:moddep1}. Otherwise, the proof proceeds identically to the proof of Proposition~\ref{prop:conjgam}. Since $n=1$, one obtains that the function $f(z)$ is multiplicatively dependent modulo constants. One therefore applies at the last step the fact that $f$ is non-constant (rather than Condition~\ref{conj:ind}), in order to obtain the necessary contradiction.
\end{proof}

\section{The Zilber--Pink conjecture in the finite rank case}\label{sec:ZPgam}

\subsection{The Zilber--Pink setting}\label{subsec:ZPgam}
We now put Theorem~\ref{thm:gamdep} in the context of the Zilber--Pink conjecture. 

For $n,k \geq 1$, we set $X =X_{n,n+k} = Y(1)^n \times \G^{n+k}$, $U = U_{n,n+k}= \h^n \times \C^{n+k}$, and let $\pi \colon U \to X$ be given by 
\[\pi(z_1, \ldots, z_n, u_1, \ldots, u_{n+k}) = (j(z_1), \ldots, j(z_n), \exp(u_1), \ldots, \exp(u_{n+k})).\]
We define (weakly) special subvarieties of $U, X$ as in Definition~\ref{def:special}. An atypical component of a subvariety $V \subset X$ is defined as in Definition~\ref{def:atyp}.

For the remainder of this paper, we fix $f \in \alg(j)$ non-constant and write $f(z) = R(j(z))$ for some rational function $R$ with algebraic coefficients. We consider the Zilber--Pink conjecture for the family of subvarieties given by
\[V_{n,k,\bar{b}} = \{(x_1, \ldots, x_n, t_1, \ldots, t_n, b_1, \ldots, b_k) \in X_{n, n+k} : t_i = R(x_i) \mbox{ for } i=1, \ldots, n\},\]
where $n, k \geq 1$ and $\bar{b} = (b_1, \ldots, b_k) \in \G^k$. The statement of the Zilber--Pink conjecture is as follows.

\begin{conj}[Zilber--Pink conjecture]\label{conj:ZPgam}
	Let $n, k \geq 1$ and $\bar{b}=(b_1, \ldots, b_k) \in \G^k$. Then there are only finitely many maximal atypical components of $V_{n,k,\bar{b}}$.
\end{conj}

 Suppose $\Gamma \leq \G$ is a subgroup of finite, positive rank. If $f$ satisfies Condition~\ref{conj:ind}, then Conjecture~\ref{conj:ZPgam} implies the finiteness of $n$-tuples of distinct $f$-special points that are $\Gamma$-dependent and minimal for this property, i.e. Proposition~\ref{prop:conjgam}. To show this, we first prove the following lemma via an easy modification of the proof of Lemma~\ref{lem:atyp}.

\begin{lem}\label{lem:gamatyp}
Suppose that $f$ satisfies Condition~\ref{conj:ind} and $\Gamma \leq \G$ is a subgroup of finite, positive rank. Then there exists $k \geq 1$ and $\bar{b} \in \G^k$ with the following property: an $n$-tuple $\sigma = (\sigma_1, \ldots, \sigma_n)$ of distinct $f$-special points such that the set $\{\sigma_1, \ldots, \sigma_n\}$ is minimally $\Gamma$-dependent gives rise to a point $\hat{\sigma} \in V_{n, k, \bar{b}}$ such that $\{\hat{\sigma}\}$ is a maximal atypical component of $V_{n, k, \bar{b}}$.
\end{lem}

\begin{proof}
	There is a finitely generated $\Gamma_0 \subset \Gamma$ such that, for every $\gamma \in \Gamma$, there exists $m \geq 1$ with $\gamma^m \in \Gamma_0$. Write $\Gamma_0 = \langle b_1, \ldots, b_k \rangle$ for some $b_1, \ldots, b_k \in \G$. We may assume that the set $\{b_1, \ldots, b_k\}$ is multiplicatively independent. Let $n \geq 1$. Set $X = X_{n, n+k}$ and $V = V_{n, k, \bar{b}}$, where $\bar{b} = (b_1, \ldots, b_k)$.
	
	Let $\sigma = (\sigma_1, \ldots, \sigma_n)$ be an $n$-tuple of distinct $f$-special points such that $\{\sigma_1, \ldots, \sigma_n\}$ is $\Gamma$-dependent, but no proper subset of $\{\sigma_1, \ldots, \sigma_n\}$ is $\Gamma$-dependent. Then 
	\[\sigma_1^{a_1} \cdots \sigma_n^{a_n} b_1^{\alpha_1} \cdots b_k^{\alpha_k} = 1\]
	for some integers $a_i, \alpha_i$ with the $a_i$ not all zero. There exist $j$-special points $x_1, \ldots, x_n$ such that each $\sigma_i = R(x_i)$. The point 
	\[\hat{\sigma}=(x_1, \ldots, x_n, \sigma_1, \ldots, \sigma_n, b_1, \ldots, b_k) \in V\]
	then lies in the intersection of $V$ with the special subvariety $T$ of $X$ defined by
	\[T = \{(\bar{z}, \bar{t}, \bar{u}) \colon z_i = x_i \mbox{ for } i=1,\ldots,n \mbox{ and } \, t_1^{a_1} \cdots t_n^{a_n} u_1^{\alpha_1} \cdots u_k^{\alpha_k}=1 \}.\]
	Here $\codim T = n+1$ and so $\dim T = n+k -1$.  Thus,
	\[\dim V + \dim T - \dim X = n + (n+k -1) - (2n +k) = -1.\]
	So $\sigma$ gives rise to an atypical component $\{\hat{\sigma}\}$ of $V$.
	
	Since Condition~\ref{conj:ind} holds for $f$ and the $b_i$ are multiplicatively independent, one may then show, as in the proof of Lemma~\ref{lem:atyp}, that this atypical point $\hat{\sigma}$ cannot be contained in any atypical component of $V$ of positive dimension. 
\end{proof}

The desired result then follows straightforwardly. 

\begin{prop}\label{prop:ZPgamimplies}
	Suppose Condition~\ref{conj:ind} and Conjecture~\ref{conj:ZPgam} hold for $f$. Let $\Gamma \leq \G$ be a subgroup of finite, positive rank. Then, for each $n \geq 1$, there are only finitely many $n$-tuples $(\sigma_1, \ldots, \sigma_n)$ of distinct $f$-special points such that the set $\{\sigma_1, \ldots, \sigma_n\}$ is $\Gamma$-dependent, but no proper subset of $\{\sigma_1, \ldots, \sigma_n\}$ is $\Gamma$-dependent.
\end{prop}

\begin{proof}
Let $\Gamma_0 \subset \Gamma$ be a finitely generated subgroup such that, for every $\gamma \in \Gamma$, there exists $n \geq 1$ such that $\gamma^n \in \Gamma_0$. Let $\{b_1, \ldots, b_k\}$ be a multiplicatively independent set of generators for $\Gamma_0$. Set $\bar{b} = (b_1, \ldots, b_k) \in \G^k$.

Suppose then that $f$ satisfies Condition~\ref{conj:ind} and Conjecture~\ref{conj:ZPgam}. Let $n \geq 1$. Then Lemma~\ref{lem:gamatyp} shows that any $n$-tuple $\sigma = (\sigma_1, \ldots, \sigma_n)$ of distinct $f$-special points such that the set $\{\sigma_1, \ldots, \sigma_n\}$ is $\Gamma$-dependent and minimal for this property gives rise to a maximal atypical component $\{\hat{\sigma}\}$ of the subvariety $V_{n, k, \bar{b}}$. Further, distinct $n$-tuples $\sigma$ of this kind give rise, in this way, to distinct points $\hat{\sigma} \in V_{n, k, \bar{b}}$. By Conjecture~\ref{conj:ZPgam}, there are only finitely many maximal atypical components of $V_{n, k, \bar{b}}$. Hence, there are only finitely many such $n$-tuples.
\end{proof}

In particular, Theorem~\ref{thm:gamdep} would follow from Conjecture~\ref{conj:ZPgam}. However, we are not able to prove Conjecture~\ref{conj:ZPgam} in general. For $n=1$ though we can prove Conjecture~\ref{conj:ZPgam}. The proof uses Proposition~\ref{prop:gam1}. When $n=2$, we show that Conjecture~\ref{conj:ZPgam} holds if $f$ satisfies Condition~\ref{conj:ind} and also a suitable finiteness assumption for points satisfying certain special conditions is true.

\subsection{The Zilber--Pink conjecture for $n=1$}\label{subsec:ZPgam1}

\begin{prop}\label{prop:ZPgam1}
	Conjecture~\ref{conj:ZPgam} holds for $n=1$.
\end{prop}

\begin{proof}
	Fix $k \geq 1$ and $\bar{b}=(b_1, \ldots, b_k) \in \G^k$. We look for the atypical components of the subvariety
	\[V = \{(x, t, \bar{b}) \colon t=R(x)\} \subset X = Y(1) \times \G^{1+k}.\]
	So $\dim V = 1$ and $\dim X = k+2$. If the set $\{b_1, \ldots, b_k\}$ is multiplicatively dependent, then $V$ is itself atypical and hence the only maximal atypical component. So we may assume that the set $\{b_1, \ldots, b_k\}$ is multiplicatively independent; in particular, no $b_i$ is a root of unity.
	
	Consider the possible special subvarieties $T$ of $X$. Clearly $X$ cannot itself intersect $V$ atypically, so we look only at proper special subvarieties. We may write $T = T_1 \times T_2$, where $T_1$ is an special subvariety of $Y(1)$ and $T_2$ is a special subvariety of $\G^{1+k}$. 
	
	Look first at those $T$ where $T_1$ is a proper special subvariety. The condition on $T_1$ must be a fixed coordinate. So $T= \{x\} \times T_2$, where $x$ is a $j$-special point and $T_2$ is a special subvariety of $\G^{k+1}$. If $T= \{x \} \times \G^{1+k}$, then $V \cap T = \{(x, R(x), b_1, \ldots, b_k)\}$, which is not atypical since
	\[\dim V + \dim T - \dim X = 1 + (k+1) - (k+2)=0.\]
	So we must have that $T= \{x\} \times T_2$ with $T_2$ a proper special subvariety of $\G^{1+k}$. Then $\dim T \leq k$, and so the intersection $V \cap T$ is atypical if it is non-empty.
	
	Suppose for now that $T_2$ has at least one fixed coordinate (i.e. a root of unity). If $V \cap T$ is non-empty, then this fixed coordinate $\zeta$ of $T_2$ must be in the first of the $\G^{1+k}$ coordinates, since no $b_i$ is a root of unity. We then require for $V \cap T \neq \emptyset$ also that $\zeta = R(x)$, and so $\zeta$ is both $f$-special and a root of unity. By Proposition~\ref{prop:moddep1}, there are at most finitely many such points. Thus, there are only finitely many atypical components of this form.
	
	So we may suppose that $T_2$ has no fixed coordinate. Then some multiplicative relation must hold on $T_2$ since $T_2$ is proper. Then $V \cap T$ is non-empty (and hence atypical) only if $R(x), b_1, \ldots, b_k$ satisfy this multiplicative relation, which must involve $R(x)$ since the $b_i$ are multiplicatively independent. Note that $R(x)$ is an $f$-special point. By Proposition~\ref{prop:gam1} applied to the finite rank group $\Gamma = \langle b_1, \ldots, b_k \rangle$, there are only finitely many $f$-special points $\sigma$ which satisfy a multiplicative relation over $b_1 ,\ldots, b_k$. So there are only finitely many such $R(x)$ to consider and hence only finitely many such atypical components.
	
	It thus remains to consider those cases where $T_1=Y(1)$. So $T = Y(1) \times T_2$, where $T_2$ must be a proper special subvariety of $\G^{k+1}$ since $T \subsetneq X$. Then
	\[V \cap T = \{ (t, R(t), b_1, \ldots, b_k) \colon (R(t), b_1, \ldots, b_k) \in T_2\}.\]
The subvariety $T_2$ of $\G^{1+k}$ is proper, so either some coordinate is fixed in $T_2$ or a multiplicative relation holds. If $V \cap T \neq \emptyset$, then any fixed coordinate in $T_2$ must be the first coordinate since no $b_i$ is a root of unity. Hence in this case $V \cap T$ must be finite. Now suppose some multiplicative relation holds on $T_2$. Since the $b_i$ are multiplicatively independent, if $V \cap T \neq \emptyset$, then this multiplicative relation must involve the first coordinate of $T_2$. Write $t_1^{a_1} \cdots t_{k+1}^{a_{k+1}}=1$ for this relation (so $a_1 \neq 0$). The equation
\[t^{a_1} b_1^{a_2} \cdots b_k^{a_{k+1}}=1\]
has only finitely many solutions $t$. So $V \cap T$ is finite.

In either of these cases, $\dim (V \cap T) = 0$. So the components of the intersection $V \cap T$ are then atypical only if
\begin{align*}
	0 &> \dim V + \dim T - \dim X \\
	&= 1 + (1 + \dim T_2) - (k+2) \\
	&= \dim T_2 -k,
	\end{align*}
i.e. if $\dim T_2 < k$. So at least two independent conditions must hold on $T_2$. Any fixed coordinate condition must apply to the first coordinate, as otherwise $V \cap T$ would be empty since no $b_i$ is a root of unity. Since the two conditions are independent, at least one of them must therefore be a multiplicative relation. Further, if $V \cap T \neq \emptyset$, then this multiplicative relation must involve the first coordinate since the $b_i$ are multiplicatively independent. This then rules out the possibility of the first coordinate of $T_2$ being fixed because that would then imply the existence of a multiplicative relation among the $b_i$ if $V \cap T \neq \emptyset$. So we must have that the second condition is also a multiplicative relation, and this multiplicative relation must again involve the first coordinate of $T_2$ for the same reason as before. So we have $T_2$ defined by conditions
\[t^{a} t_1^{a_1} \cdots t_k^{a_k} =1,\]
\[t^{a'} t_1^{a_1'} \cdots t_k^{a_k'} =1\]
with $a, a' \neq 0$. If $V \cap T \neq \emptyset$, then there is some $t$ such that
\[t^{a} b_1^{a_1} \cdots b_k^{a_k} =1,\]
\[t^{a'} b_1^{a_1'} \cdots b_k^{a_k'} =1.\]
Since the two multiplicative conditions are independent we may then eliminate $t$ to get a multiplicative dependence among $b_1, \ldots, b_k$, a contradiction. Hence there are no atypical components of this form, and the proof is complete.
\end{proof}

\subsection{The Zilber--Pink conjecture for $n=2$}\label{subsec:ZPgam2}

In this subsection, we show that Conjecture~\ref{conj:ZPgam} for $V_{2,k,\bar{b}}$ follows from Condition~\ref{conj:ind} together with a finiteness statement for pairs $(x_1, x_2) \in Y(1)^2$ satisfying certain special conditions. 

\begin{definition}\label{def:gampair}
Let $\Gamma \leq \G$. We define an $(f,\Gamma)$-pair to be a pair $(x_1, x_2) \in Y(1)^2$ such that $(x_1, x_2)$ satisfies a modular relation, $R(x_1)$ is $\Gamma$-dependent, and $R(x_2)$ is $\Gamma$-dependent.
\end{definition}

We now prove the following conditional version of Conjecture~\ref{conj:ZPgam} for $n=2$.

\begin{prop}\label{prop:condZPgam}
Suppose that $f$ satisfies Condition~\ref{conj:ind}. Let $k \geq 1$ and $\bar{b}=(b_1, \ldots, b_k) \in \G^k$. Set $\Gamma = \langle b_1, \ldots, b_k \rangle \leq \G$.  Suppose that there are only finitely many $(f, \Gamma)$-pairs $(x_1, x_2)$ with $x_1 \neq x_2$ and $x_1, x_2$ not $j$-special. Then Conjecture~\ref{conj:ZPgam} holds for $V_{2, k, \bar{b}}$.
\end{prop}

\begin{proof}
	We look for the atypical components of 
	\[V = \{(x_1, x_2, t_1, t_2, b_1, \ldots, b_k) \colon t_1 = R(x_1), t_2 = R(x_2)\} \subset X = Y(1)^2 \times \G^{2+k}.\]
	So $\dim V = 2$ and $\dim X = k+4$. A special subvariety $T$ of $X$ gives rise to an atypical component of $V$ if
	\[\dim (V \cap T) > \dim V + \dim T - \dim X = 2 - \codim T.\]
	
	If the set $\{b_1, \ldots, b_k\}$ is multiplicatively dependent, then $V$ is itself atypical and hence the only maximal atypical component. Thus we may assume that the set $\{b_1, \ldots, b_k\}$ is multiplicatively independent. The group $\Gamma$ is of finite rank and so, by Proposition~\ref{prop:conjgam}, for every $n \geq 1$ there are only finitely many $n$-tuples of distinct $f$-special points which are $\Gamma$-dependent and minimal for this property.
	
	Consider the possible special subvarieties $T$ of $X$. We find those $T$ that give rise to atypical components of $V$. We split into cases according to $\codim T$. 
	
	\begin{enumerate}[wide, labelwidth=!, labelindent=0pt]
		
		\item Clearly $X$ cannot itself intersect $V$ atypically, so we only need to look at proper special subvarieties. \\[\parskip]
		
		\item Suppose $\codim T = 1$. Then some component of $V \cap T$ is atypical if and only if $\dim (V \cap T) \geq 2$. If $T$ is defined by one fixed multiplicative coordinate, then $\dim (V \cap T)$ is either $0$ or $1$ depending on whether this fixed coordinate is one of the first two $\G$ coordinates or not. In either case, the intersection is not atypical. If $T$ is defined by specifying one fixed modular coordinate, then $\dim (V \cap T)=1$ and the intersection is not atypical. Similarly if $T$ is defined by a single modular relation. So the only case left is where $T$ is defined by a multiplicative relation. Then, since the $b_i$ are multiplicatively independent, this multiplicative relation must involve at least one of the first two $\G$ coordinates if $V \cap T \neq \emptyset$. But then $\dim (V \cap T )= 1$ and the intersection is not atypical.\\[\parskip]
		
		\item Next we look at $\codim T=2$. Then a component of $V \cap T$ is atypical if and only if it is positive dimensional. Clearly then $T$ cannot be defined by two independent modular conditions. Suppose $T$ is defined by two independent multiplicative conditions (either fixed coordinates or multiplicative relations). If $V \cap T \neq \emptyset$, then each of these conditions must involve at least one of the first two $\G$ coordinates since the $b_i$ are multiplicatively independent. In all such cases, one then sees that $V \cap T$ must be finite and hence its components cannot be atypical. 
		
		So $T$ must be defined by one modular condition and one multiplicative condition. If both conditions are fixed coordinates, then $V \cap T$ is positive dimensional only if the two conditions either both apply to the respective first coordinate or both apply to the respective second coordinate, and they also satisfy $\zeta = R(x)$, where $\zeta$ is the multiplicative fixed coordinate and $x$ is the fixed modular coordinate. In such cases, $\zeta$ is both an $f$-special point and a root of unity, and so by Proposition~\ref{prop:conjtup} there are only finitely many atypical components of this kind. 
		
		If $T$ is defined by a modular relation and a multiplicative relation, then, by Condition~\ref{conj:ind} and the multiplicative independence of the $b_i$, the intersection $V \cap T$ cannot be positive dimensional unless the multiplicative relation has the form $x=y$. Since no non-constant modular function is invariant under a larger subgroup of $\GL$ than $\Q^{\times} \cdot \SL$, the modular relation must also be of the form $x=y$. So there is just one such atypical component.
		
		Suppose $T$ is defined by a modular relation and a fixed multiplicative coordinate. Then this fixed coordinate must be one of the first two $\G^{2+k}$ coordinates. But then the modular relation on the $Y(1)^2$ coordinates implies that $V \cap T$ is finite and so not atypical.
		
		If $T$ is defined by a multiplicative relation and a fixed modular coordinate, then $V \cap T$ is positive dimensional only if the multiplicative relation involves the $\G^{2+k}$ coordinate corresponding to the fixed $Y(1)^2$ coordinate, but not the other of the first two $\G^{2+k}$ coordinates. If the fixed modular coordinate is given by the $j$-special point $x$, then the $f$-special point $R(x)$ is $\Gamma$-dependent since it satisfies a multiplicative relation over the $b_i$. Hence, by Proposition~\ref{prop:conjgam}, there are only finitely many atypical components of this form to consider.\\[\parskip]
		
		\item When $\codim T =3$, the components of the intersection $V \cap T$ are atypical if and only if $V \cap T \neq \emptyset$. If all three conditions defining $T$ are fixed coordinates, then $V \cap T \neq \emptyset$ implies that $V \cap T$ must already be contained in one of the positive dimensional atypical components arising from $T$ defined by a fixed modular coordinate and a fixed multiplicative coordinate. So we may assume that at least one condition defining $T$ is a relation. 
		
		If there were three independent multiplicative conditions defining $T$, then $V \cap T \neq \emptyset$ would imply a multiplicative relation among the $b_i$, which is impossible. Clearly one cannot have three independent conditions on the $Y(1)^2$ coordinates. So there must be at least one modular condition and at least one multiplicative condition defining $T$.
		
		If there are two independent modular conditions, then we may assume these are fixed coordinates. The other condition must then be a multiplicative relation, and this relation must involve at least one of the first two $\G^{2+k}$ coordinates if $V \cap T \neq \emptyset$ since the $b_i$ are multiplicatively independent. If the multiplicative relation involves only one of the first two $\G^{2+k}$ coordinates, then the corresponding fixed modular coordinate $x$ gives rise to a $\Gamma$-dependent $f$-special point $R(x)$. We are thus in one of the positive-dimensional atypical components arising from a special subvariety of codimension $2$. 
		
		If the multiplicative relation involves both the first two $\G^{2+k}$ coordinates, then the fixed modular coordinates $x_1, x_2$ give rise to a $\Gamma$-dependent pair of $f$-special points $(R(x_1), R(x_2))$. If some $R(x_i)$ is $\Gamma$-dependent, then we are in one of the previously identified atypical components. We may thus assume that no subtuple of $(R(x_1), R(x_2))$ is $\Gamma$-dependent. Since there are only finitely many pairs of $j$-special points $(x_1, x_2)$ with $x_1 \neq x_2$ and $R(x_1)=R(x_2)$, as shown in the proof of Proposition~\ref{prop:ZP2}, we may also assume that $R(x_1) \neq R(x_2)$. It thus follows from Proposition~\ref{prop:conjgam} that there are at most finitely many maximal atypical components of this form.
		
		So we may suppose that $T$ is defined by two multiplicative conditions and one modular condition. If the multiplicative conditions are both fixed coordinates, then the modular condition is a relation. We thus have that $V \cap T = \{(x_1, x_2, \zeta_1, \zeta_2, b_1, \ldots, b_k)\}$, where $(x_1, x_2)$ satisfies a modular relation, $\zeta_1, \zeta_2$ are roots of unity, and $\zeta_i = R(x_i)$ for $i=1,2$. By Proposition~\ref{prop:modpairs}, there are only finitely many such components satisfying the additional restrictions $x_1 \neq x_2$ and neither $x_1$ nor $x_2$ is $j$-special. If these additional restrictions are not met, then we are in one of the previously identified atypical components. 
		
		So we may assume that at least one of the multiplicative conditions is a relation; clearly, this relation must involve at least one of the first two $\G^{2+k}$ coordinates. Suppose the modular condition is a fixed coordinate $x$. If the second multiplicative condition is a fixed coordinate $\zeta$, then clearly this must be one of the first two $\G^{2+k}$ coordinates. Further since $\zeta$ is a root of unity, we can eliminate this coordinate from the multiplicative relation. If $V \cap T \neq \emptyset$, then, according to whether the two fixed coordinates are in the same respective position or not, either $\zeta = R(x)$ is both a root of unity and $f$-special or $(R(x), b_1, \ldots, b_k)$ satisfy the multiplicative relation and so $R(x)$ is a $\Gamma$-dependent $1$-tuple. Thus such components are contained in positive dimensional atypical components arising from special subvarieties of codimension $2$.
		
		So now suppose the modular condition is a fixed coordinate $x$ and both the multiplicative conditions are relations. If $V \cap T \neq \emptyset$, then both relations must involve at least one of the first two $\G^{2+k}$ coordinates. They cannot both involve only the first (respectively the second) of the first two $\G^{2+k}$ coordinates, since by their independence we would then be able to obtain a relation among the $b_i$. If both relations involve both the first two $\G^{k+2}$ coordinates, then we may eliminate either of these two coordinates. Thus we may assume that the first relation involves the first but not the second $\G^{2+k}$ coordinate and that the second relation involves the second but not the first $\G^{2+k}$ coordinate. The components of $V \cap T$ are therefore contained in some of the already identified positive dimensional atypical components.
		
		We thus reduce to considering when the modular condition is a relation and at least one of the multiplicative conditions is a relation. Since the $b_i$ are multiplicatively independent, if $V \cap T \neq \emptyset$, then any component of $V \cap T$ must have the form $\{(x_1, x_2, R(x_1), R(x_2), b_1, \ldots, b_k)\}$ where $(x_1, x_2)$ satisfies a modular relation and $R(x_1), R(x_2)$ are both (individually) $\Gamma$-dependent. Any such $(x_1, x_2)$ is an $(f, \Gamma)$-pair. The finiteness of the resulting maximal atypical components then follows from the assumption on $(f, \Gamma)$-pairs in the hypotheses of the proposition. This is because if $x_1 = x_2$ or some $x_i$ is $j$-special, then the corresponding component of $V \cap T$ is already contained in one of the atypical components above.
		\\[\parskip]
		
		\item We now consider the case when $\codim T \geq 4$. If $T$ is defined by $\geq 3$ independent multiplicative conditions and $V \cap T \neq \emptyset$, then we can eliminate the first two $\G^{2+k}$ coordinates from these relations and obtain a multiplicative dependency among $b_1, \ldots, b_k$. This obviously cannot happen. Clearly, there can also be no more than two independent  modular conditions defining $T$. Thus the only case to consider is when $T$ is defined by two modular conditions and two multiplicative conditions.  
		
		The components of the intersection $V \cap T$ are then atypical if and only if $V \cap T \neq \emptyset$. We may assume that the two modular conditions are both fixed coordinates. If one of the multiplicative conditions is a fixed coordinate, then this must be one of the first two $\G^{2+k}$ coordinates. If $V \cap T \neq \emptyset$, then the root of unity corresponding to this fixed coordinate must also be an $f$-special point because the respective modular coordinate is $j$-special. In this case, the intersection is contained in one of the already identified positive dimensional atypical components. Therefore we may assume that both the multiplicative conditions are relations. The components that can arise here are thus all contained in larger atypical components identified in (3).	\qedhere
	\end{enumerate}
\end{proof}

We note here the similarity between the finiteness assumption on $(f, \Gamma)$-pairs contained in the hypotheses of Proposition~\ref{prop:condZPgam} and the finiteness statement for modular--torsion tuples in Proposition~\ref{prop:modpairs}. In both cases, one is dealing with the atypical components which arise from one modular relation and two independent multiplicative relations. However, in the case of Proposition~\ref{prop:modpairs} we are able to prove finiteness, whereas in Proposition~\ref{prop:condZPgam} we must assume it.

The difference is that in Proposition~\ref{prop:modpairs} the points $R(x_1), R(x_2)$ are roots of unity, since if $R(x_1), R(x_2)$ satisfy two independent multiplicative relations then they must both be roots of unity. One thus obtains that $x_1, x_2$ are algebraic and their heights are bounded, which is crucial for the finiteness proof. In contrast, for an $(f, \Gamma)$-pair $(x_1, x_2)$ the multiplicative relations satisfied by $R(x_1), R(x_2)$ involve also the generators $b_1, \ldots, b_k$ of $\Gamma$. One thus obtains only the weaker condition of each $R(x_i)$ being $\Gamma$-dependent, rather than a root of unity. In particular, we thus do not seem able to obtain suitable bounds on the heights of $x_1, x_2$ (they might not even be algebraic) in order to prove the finiteness of the pairs $(x_1, x_2)$ by a modified version of the argument for Proposition~\ref{prop:modpairs}. 

\bibliographystyle{amsplain}
\bibliography{refs}

\end{document}